\documentclass[a4paper,11pt,reqno]{customart}

\usepackage[utf8x]{inputenc}
\usepackage[margin=1in]{geometry}
\usepackage{verbatim}
\usepackage{color}
\usepackage[table]{xcolor}
\usepackage{listings}
\usepackage{amssymb,amsfonts,amsthm,amsmath}
\usepackage{url}
\usepackage{enumerate}
\usepackage{todonotes}
\usepackage[all,cmtip]{xy}
\usepackage{multirow}
\usepackage{attachfile}
\usepackage{stmaryrd}
\usepackage{footnote}
\makesavenoteenv{tabular}
\usepackage{hyperref}
\usepackage{cleveref}

\makeatletter
\newcommand\footnoteref[1]{\protected@xdef\@thefnmark{\ref{#1}}\@footnotemark}
\makeatother

\hypersetup{
	pdfborder={0 0 0}
}

\renewcommand{\AA}{\mathbb{A}}

\newcommand{\A}{\mathcal{A}}
\newcommand{\B}{\mathcal{B}}
\newcommand{\C}{\mathcal{C}}

\newcommand{\M}{\mathcal{M}}

\newcommand{\F}{\mathcal{F}}

\newcommand{\V}{\mathcal{V}}
\newcommand{\U}{\mathcal{U}}

\newcommand{\W}{\mathcal{W}}

\renewcommand{\P}{\mathcal{P}}

\renewcommand{\L}{\mathcal{L}}

\newcommand{\uh}{\upharpoonright}

\def\qt#1{``#1''}%

\newtheoremstyle{custom}
  {10pt}
  {10pt}
  {\normalfont}
  {}
  {\bfseries}
  {}
  { }
  {}

\theoremstyle{custom}

\usepackage{xcolor}
\usepackage{soul}

\newtheorem{theorem}{Theorem}[section]
\newtheorem{lemma}[theorem]{Lemma}

\newtheorem{corollary}[theorem]{Corollary}
\newtheorem{proposition}[theorem]{Proposition}

\theoremstyle{definition}
\newtheorem{definition}[theorem]{Definition}
\newtheorem{example}[theorem]{Example}

\theoremstyle{remark}
\newtheorem{remark}[theorem]{Remark}

\numberwithin{equation}{section}

\newtheoremstyle{noparens}%
  {}{}%
{}{}%
{\bfseries}{.}%
{ }%
{\thmname{#1}\thmnumber{ #2}\thmnote{ #3}}
\theoremstyle{noparens}
\newtheorem*{theorem*}{Theorem}

\title{Partition genericity and pigeonhole basis theorems}
\author{
  Benoit Monin \and Ludovic Patey
}

\begin{document}

\begin{abstract}
There exist two notions of typicality in computability theory, namely, genericity and randomness. In this article, we introduce a new notion of genericity, called \emph{partition genericity}, which is at the intersection of these two notions of typicality, and show that many basis theorems apply to partition genericity. More precisely, we prove that every co-hyperimmune set and every Kurtz random is partition generic, and that every partition generic set admits a weak infinite subsets. In particular, we answer a question of Kjos-Hanssen and Liu by showing that every Kurtz random admits an infinite subset which does not compute any set of positive Hausdorff dimension.
Partition genericty is a partition regular notion, so these results imply many existing pigeonhole basis theorems.
\end{abstract}

\maketitle

\section{Introduction}\label[section]{sect:introduction}

The infinite pigeonhole principle can be considered as the most basic statement from Ramsey's theory. The infinite pigeonhole principle for 2 colors can be formulated as \qt{for every set $A \subseteq \omega$, there is an infinite set $H \subseteq A$ or $H \subseteq \overline{A}$.} From a combinatorial viewpoint, the infinite pigeonhole principle is trivial. On the other hand, the computational analysis of this principle is very subtle and received the attention of the computability community for decades, motivated by the reverse mathematics of Ramsey's theorem for pairs. 

\subsection{Pigeonhole basis theorems}


The computability-theoretic analysis of a mathematical problem consists in understanding, given an instance, how computably complicated are its solutions. From this perspective, a lower bound is a statement of the form \qt{There exists an instance such that every solution is computationally strong}, while an upper bound is of the form \qt{For every instance, there is a computationally weak solution.} Here, the notions of strength and weaknesses range over many computability-theoretic properties.

A \emph{pigeonhole basis theorem} is an upper bound for the pigeonhole principle, that is, a statement of the form: \qt{For every set $A \subseteq \omega$, there is an infinite set $H \subseteq A$ or $H \subseteq \overline{A}$ which is computationally weak}. Several pigeonhole basis theorems have been proven: 
\begin{enumerate}
	\item If $B$ is a non-computable set, then for every set $A \subseteq \omega$, there is an infinite set $H \subseteq A$ or $H \subseteq \overline{A}$ such that $B \not \leq_T H$ (Dzhafarov and Jockusch~\cite{Dzhafarov2009Ramseys}).
	\item If $B$ is a non-$\Sigma^0_1$ set, then for every set $A \subseteq \omega$, there is an infinite set $H \subseteq A$ or $H \subseteq \overline{A}$ such that $B$ is not $\Sigma^0_1(H)$ (Wang~\cite{Wang2016definability}).
	\item If $f$ is hyperimmune, then for every set $A \subseteq \omega$, there is an infinite set $H \subseteq A$ or $H \subseteq \overline{A}$ such that $f$ is $H$-hyperimmune (Patey~\cite{patey2015iterative}).
	\item For every set $A \subseteq \omega$, there is an infinite set $H \subseteq A$ or $H \subseteq \overline{A}$ of non-PA degree (Liu~\cite{Liu2012RT22}).
	\item For every set $A \subseteq \omega$, there is an infinite set $H \subseteq A$ or $H \subseteq \overline{A}$ of non-random degree, and which does not compute any set of positive Hausdorff dimension (Liu~\cite{liu2015cone}).
\end{enumerate}

Asking for an infinite subset of $A$ or of $\overline{A}$ is important, as there exist sets $A$ such that every infinite subset is arbitrarily strong. For example, if $A$ is the set of all initial segments of a set $B$, then every infinite subset of $A$ computes $B$. In some cases however, one can fix the side of the subset. This happens in particular when the set $A$ is sufficiently typical, where typicality means either randomness or genericity.


\subsection{Randomness subset basis theorems}

Randomness is a notion of typicality which was originally defined using measure theory. More recently, Algorithmic Randomness gave a formal meaning to the notion of \emph{random sequence} using effective measure theory and Kolmogorov complexity. This yielded a hierarchy randomness notions, among which we should mention (in increasing order) Kurtz randomness, Schnorr randomness, Martin-L\"of randomness, Schnorr 2-randomness, and 2-randomness.

A \emph{Randomness subset basis theorem} is a statement of the form: \qt{For every sufficiently random sequence $A \subseteq \omega$, there is an infinite subset set $H \subseteq A$ which is computationally weak}. Here, by \qt{sufficiently random}, we mean that the class of all such sets~$A$ has positive measure. One can then quantify the amount of randomness needed for such statement, and obtain a theorem of the form \qt{For every $\boxdot$ random sequence~$A \subseteq \omega$, there is an infinite subset set $H \subseteq A$ which is computationally weak}, where  \qt{$\boxdot$ random} should be replaced by the right notion of randomness, such as Martin-L\"of randomness for example. Subsets of random sequences were mainly studied as mass problems. For example, the Muchnik degree of the class of infinite subsets of Martin-L\"of random sequences (seen as set of numbers), is the Muchnik degree of DNC functions (Kjos-Hanssen~\cite{kjoshanssen2009infinite}, Greenberg and Miller~\cite{greenberg2009lowness}). However, a few randomness subset basis theorems appeared in the literature:

\begin{enumerate}
	\item Every 2-random (or even Schnorr 2-random) has an infinite subset which does not compute a 1-random (Kjos-Hanssen~\cite{kjoshanssen2011strong}).
	\item Every 1-random has an infinite subset which does not compute a 1-random, or even which does not compute any set with positive Hausdorff dimension (Kjos-Hanssen and Liu~\cite{kjoshanssen2020extracting}).
\end{enumerate}

Kjos-Hanssen and Liu~\cite{kjoshanssen2020extracting} asked whether these results could be improved to weaker notions of randomness, such as Schnorr randomness.
In this article, we give a strong positive answer by showing that these results can be improved to Kurtz randomness.


\subsection{Genericity subset basis theorems}

Genericity is a notion of typicality which can be defined in terms of co-meager sets. The default notion of genericity considered is Cohen genericity. Later, Jockusch studied effectivizations of Cohen genericity, yielding again a hierarchy genericity notions, among which we mention in increasing order bi-hyperimmunity, weak 1-genericity and 1-genericity.
A \emph{genericity subset basis theorem} is a statement of the form: \qt{For every sufficiently Cohen generic set $A \subseteq \omega$, there is an infinite subset set $H \subseteq A$ which is computationally weak}. Genericity subset basis theorems were not specifically studied per se. One can however mention one such result:

\begin{enumerate}
	\item If $B$ is a non-computable set, then every co-hyperimmune set has an infinite subset which does not compute~$B$ (Hirschfeldt et al.~\cite{Hirschfeldt2008strength}).
\end{enumerate}

\subsection{Partition genericity}

In this article, we define a new notion of genericity, called \emph{partition genericity}, and prove many statements of the form \qt{Every partition generic set $A$ has an infinite computationally weak subset.} We call these statements \emph{partition genericity subset basis theorems}. Contrary to Martin-L\"of randomness and Cohen genericity, this notion of partition genericity enjoys a property that one would expect of a subset basis theorem, that is, partition genericity is closed under supersets.


\begin{theorem}\label[theorem]{thm:pg-basis-theorem}\ 
\begin{enumerate}
	\item If $B$ is a non-computable set, and $A$ is partition generic, then there is an infinite set $H \subseteq A$ such that $B \not \leq_T H$.
	\item If $B$ is a non-$\Sigma^0_1$ set, and $A$ is partition generic relative to~$B$, then there is an infinite set $H \subseteq A$ such that $B$ is not $\Sigma^0_1(H)$.
	\item If $f$ is hyperimmune, and $A$ is partition generic relative to~$f$, then there is an infinite set $H \subseteq A$ such that $f$ is $H$-hyperimmune.
	\item If~$A$ is partition generic, then there is an infinite set $H \subseteq A$ of non-PA degree.
	\item If~$A$ is partition generic, then there is an infinite set $H \subseteq A$ of non-random degree, and furthermore, which does not compute any set of positive Hausdorff dimension.
\end{enumerate}	
\end{theorem}

In particular, every co-hyperimmune set and every Kurtz random is partition generic. Moreover, we show that partition genericity is almost a partition regular notion (see \Cref{cor:partition-generic-regular}). It follows that all these partition genericity subset basis theorems imply all the pigeonhole basis theorems mentioned above.

\subsection{Organization of this paper}

In \Cref{sect:pr}, we introduce the central notion of partition regularity, and study it both from a combinatorial and a computability-theoretic viewpoint. We then define the notion of partition genericity. Then, in \Cref{sect:applications}, we prove a first range of applications which hold for unrelativized partition genericity. The two next sections, \Cref{sect:pres-hyp} and \Cref{sect:pres-sigma} are devoted to the preservation of hyperimmunity and non-$\Sigma^0_1$ definitions, respectively. These subset basis theorems require the definition of alternative notions of genericity. Last, we study lowness for various notions related to partition genericity in \Cref{sect:questions}.

\subsection{Notation}

We use lower case letters $a, b, c$ for integers, upper case letters for sets of integers, and rounded letters $\A, \B$ for classes.

A \emph{$k$-cover} of a set $X$ is a $k$-tuple of sets $X_0, \dots, X_{k-1}$
such that $X_0 \cup \dots \cup X_{k-1} \supseteq X$. We do not require the sets $X_i$ to be pairwise disjoint.
Given a set~$X \subseteq \omega$ and some~$n \in \omega$, we let~$[X]^n$ denote the set of all subsets of~$X$ of size~$n$. Accordingly, we write $[X]^\omega$ for the class of all infinite subsets of~$X$.
We write~$2^{=n}$ for the set of all binary strings of length~$n$, and $2^{<\omega} = \bigcup_n 2^{=n}$. We write $|\sigma|$ for the length of the string~$\sigma$.

A \emph{Mathias condition} is a pair $(\sigma, X)$, where $\sigma \in 2^{<\omega}$ is a finite binary string, $X$ is an infinite set of integers, and $\min X > |\sigma|$.

\section{Partition regularity}\label[section]{sect:pr}

In this section, we conduct a general study of partition regularity from a computability-theoretic viewpoint. We introduce several related concepts, including the central notion of partition genericity, which will be justified by the constructions of \Cref{sect:applications}.

\subsection{Partition regularity}

The central notion we are going to consider in this article is the one of partition regularity. 
This concept comes from Ramsey theory and can be considered as a generalization of the infinite pigeonhole principle.

\begin{definition}
A \emph{partition regular} class is a collection of sets $\L \subseteq 2^\omega$ such that:
\begin{enumerate}
\item $\L$ is not empty
\item If $X \in \L$ and $X \subseteq Y$, then $Y \in \L$
\item For every~$k$, if $X \in \L$ and $Y_0 \cup \dots \cup Y_k \supseteq X$, then there is $i \leq k$ such that $Y_i \in \L$
\end{enumerate}
\end{definition}

Ramsey's theory is sometimes characterized as the study of which classes are partition regular. 
There are many well-known examples of partition regular classes in combinatorics:

\begin{example}The following classes are partition regular.
\begin{enumerate}
	\item $\{ X \subseteq \omega : X \mbox{ is infinite} \}$ by the infinite pigeonhole principle ;
	\item $\{X \subseteq \omega :\ n \in X\}$ for a fixed~$n$ ;
	\item $\{X \subseteq \omega : \limsup_{n \rightarrow \infty} \frac{| \{1,2,\ldots,n\} \cap X|}{n} > 0 \}$ ;
	\item $\{X \subseteq \omega :\ \sum_{n \in X} \frac{1}{n} = \infty\}$.
\end{enumerate}
\end{example}

In the computability-theoretic realm, many pigeonhole basis theorems can be rephrased as statements about partition regularity.

\begin{example}The following classes are partition regular.
\begin{enumerate}
	\item $\{ X \subseteq \omega : \exists Y \in [X]^\omega\ Y \not \geq_T C \}$ for any $C \not \leq_T \emptyset$ (see Dzhafarov and Jockusch~\cite{Dzhafarov2009Ramseys})
	\item $\{ X \subseteq \omega : \exists Y \in [X]^\omega\ Y \mbox{ is not of PA degree} \}$ (see Liu~\cite{Liu2012RT22})
	\item $\{ X \subseteq \omega : \exists Y \in [X]^\omega\ Y^{(n)} \not \geq_T C \}$ for any non-$\Delta^0_{n+1}$ set $C$ (see Monin and Patey~\cite{monin2021weakness}) ;
	\item $\{ X \subseteq \omega : \exists Y \in [X]^\omega\ \omega_1^Y = \omega_1^{ck} \}$ (see Monin and Patey~\cite{monin2021weakness}).
\end{enumerate}
\end{example}

Dorais~\cite{dorais2012variant} was the first to use partition regular classes in the context of reverse mathematics. More precisely, he worked with a variant of Mathias forcing whose reservoirs avoid a $\Sigma^0_2$ free ideal over $2^\omega$. A class is a free ideal iff it is the complement of a partition regular class.

\subsection{Non-trivial classes}

A partition regular class can be though of as a notion of largeness. Indeed, if we interpret $X \in \L$ as \qt{$X$ is large}, then the axioms of partition regularity say that if a set is large, then any superset of it is large, and if we split a large set into two (or finitely many) parts, then at least one of the parts is large. There exist however a family of partition regular classes that fails this intuition. We call them principal classes.

\begin{definition}
A partition regular class $\L \subseteq 2^\omega$ is \emph{principal} if $\L = \{X \in 2^{\omega}\ :\ n \in X\}$ for some $n$. A partition regular class $\L$ is \emph{non-trivial} if it contains only infinite sets, otherwise it is \emph{trivial}.
\end{definition}

The following proposition shows that once one excludes the principal partition regular classes, then the remaining partition regular classes satisfy at least one enjoyable property of largeness, namely, having only infinite elements.

\begin{proposition}\label[proposition]{prop:non-trivial-principal}
A partition regular class $\L$ is non-trivial iff it contains no principal partition regular subclass.
\end{proposition}
\begin{proof}
It is clear that if $\L$ is non-trivial, it does not contain a principal partition regular subclass. Suppose now $\L$ is trivial, that is, $\L$ contains a finite set $X = \{n_1, \dots, n_k\}$. Then in particular we have $\{n_1\} \cup \dots \cup \{n_k\} \supseteq X$. It follows that we must have $\{n_i\} \in \L$ for some $i \leq k$. Then any set $X$ containing $n_i$ is in $\L$, that is, we have $\{X \in 2^{\omega}\ :\ n_i \in X\} \subseteq \L$.
\end{proof}

Non-trivial partition regular classes admit a simple characterization, which will be later used to generalize the concept of non-triviality to arbitrary classes.

\begin{proposition}\label[proposition]{prop:non-trivial-u2}
A partition regular class $\L$ is non-trivial iff $\L$ is included in the $\Sigma^0_1$ class $\U_2$ of sets containing at least two distinct elements.
\end{proposition}
\begin{proof}
If $\L$ is non-trivial then every member of $\L$ is infinite and clearly $\L \subseteq \U_2$. If $\L$ is trivial then by \Cref{prop:non-trivial-principal}, it contains $\{n\}$ for some $n \in \omega$ and thus we do not have $\L \subseteq \U_2$.
\end{proof}

The following proposition says that any non-trivial partition regular class must contain many elements, in a measure-theoretic sense. All the partition regular classes we are going to consider in the applications are Borelian, hence measurable.

\begin{proposition} \label{prop:measure:pr}
Let $\L$ be a non-trivial partition regular class. Then $\L$ is closed by finite change of its elements. Furthermore if $\L$ is measurable it has measure $1$.
\end{proposition}
\begin{proof}
Let $X \in \L$. By definition, any $Y \supseteq X$ also belongs to $\L$. Thus $\L$ is closed by finite addition of elements. Consider now any $Y \subseteq X$ such that $|X - Y|$ is finite. In particular, $X = Y \cup \{n_0, \dots, n_k\}$ for some elements $n_0, \dots, n_k$. As $\L$ contains only infinite elements, we must have $Y \in \L$. Thus $\L$ is closed by finite suppression. We easily conclude that $\L$ is closed by finite changes.

If $\L$ is measurable, by Kolmogorov 0-1 law, $\L$ is either of measure~$0$ or of measure~$1$. Suppose for contradiction that $\L$ is of measure $0$. As $\L$ is measurable, if must be included in some Borel set $\A$ of measure $0$. Let $O$ be an oracle such that $\A$ is included in a $\Pi^0_2(O)$ set effectively of measure $0$. Then no element of $\L$ is $O$-Martin-L\"of random. Let $Z$ be any $O$-Martin-L\"of random set. We also have that $\overline{Z}$ is  $O$-Martin-L\"of random. Also $\omega \subseteq Z \cup \overline{Z}$. As $\omega \in \L$ we must have $Z \in \L$ or $\overline{Z} \in \L$, which is a contradiction. Thus $\L$ is not of measure $0$ and therefore it is of measure $1$.
\end{proof}

\subsection{Closure properties}

We now study some good closure properties enjoyed by the collection of all partition regular classes. A superclass of partition regular class is not partition regular in general, even when the superclass is closed under superset. For example, let $A$ be a bi-infinite set, and let $\L_A = \{ X \in 2^\omega : |A \cap X| = \infty \}$. Then $\L_A$ is a partition regular class, but $\L = \L_A \cup \{ X \in 2^\omega : X \supseteq \overline{A} \}$  is not. Indeed, let $x_0 = \min \overline{A}$ and $B = \overline{A} \setminus \{x_0\}$. Then $\{x_0\} \cup B = \overline{A} \in \L$, but neither $\{x_0\}$, nor $B$ belong to $\L$.
On the other hand, an arbitrary union of partition regular classes is partition regular.

\begin{proposition} \label{prop:unionlarge}
Suppose $\{\L_i\}_{i \in I}$ is an arbitrary non-empty collection of partition regular classes. Then $\bigcup_{i \in I} \L_i$ is a partition regular class.
\end{proposition}
\begin{proof}
It is clear that $\bigcup_{i \in I} \L_i$ is not empty. Let $X \in \bigcup_{i \in I} \L_i$. Let $Y \supseteq X$. There is some $i \in I$ such that $X \in \L_i$. As $\L_i$ is partition regular, $Y \in \L_i \subseteq \bigcup_{i \in I} \L_i$.

Let $X \in \bigcup_{i \in I} \L_i$. Let $Y_0 \cup \dots \cup Y_k \supseteq X$. There is some $i \in I$ such that $X \in \L_i$. As $\L_i$ is partition regular, $Y_j \in \L_i \subseteq \bigcup_{i \in I} \L_i$ for some $j \leq k$.
\end{proof}

In particular for every class $\A$ containing a partition regular class, there is a largest partition regular class included in $\A$.

\begin{definition}
Given a class $\A \subseteq 2^\omega$, let $\L(\A)$ denote the largest partition regular subclass of $\A$. If $\A$ does not contain a partition regular class, let $\L(\A)$ be the empty set.
\end{definition}

The largest partition regular class included in $\A$ admits a simple explicit definition that we shall use to analyse the definitional complexity of the partition regular classes we consider.

\begin{proposition}\label{prop:complarge}
Let $\A \subseteq 2^\omega$ be any class. Then
$$\L(\A) = \{X \in 2^\omega\ :\ \forall k\ \forall X_0\cup\dots\cup X_k \supseteq X\ \exists i \leq k\ X_i \in \A\}$$
\end{proposition}
\begin{proof}
Note that by definition, $\L(\A) \subseteq \A$, as if $X \notin \A$ then itself as a $1$-cover is not in $\A$, so $X \notin \L(\A)$.
Let us show that $\L(\A)$ contains every partition regular class included in $\A$. Suppose $\L \subseteq \A$ is partition regular. Then given $X \in \L$, for every $k$ and every $X_0\cup\dots\cup X_k \supseteq X$ we have $X_i \in \L \subseteq \A$ for some $i \leq k$. It follows that $X \in \L(\A)$ and thus that $\L \subseteq \L(\A)$.

Assume $\L(\A) \neq \emptyset$. Let us show that $\L(\A)$ is partition regular.
Suppose $X \in \L(\A)$. Let $Y \supseteq X$. Then for every $k$, every $k$-cover of $Y$  is also a $k$-cover of $X$. As $X \in \L(\A)$, one element of the $k$-cover belongs to $\A$. Thus for every $k$ and every $k$-cover of $Y$, one element of the $k$-cover belongs to $\A$. It follows that $Y \in \L(\A)$.
Let $X \in \L(\A)$ and let $Y_0 \cup \dots \cup Y_k \supseteq X$ for some $k$. Let us show there is some $i \leq k$ such that $Y_i \in \L(\A)$. Suppose for contradiction that this is not the case. In particular for every $i \leq k$ there are sets $Y^i_{0}, \dots Y^i_{k_i} \supseteq Y_i$ such that $\forall j \leq k_i$, we have $Y^i_{j} \notin \A$. In particular the sets $\{Y^i_{j}\}_{i \leq k, j \leq k_i}$ are a finite cover of $X$ such that for every $i \leq k$ and every $j \leq k_i$ we have $Y^i_{j} \notin \A$. This contradicts that $X \in \L(\A)$. Thus there must exists $i \leq k$ such that $Y_i \in \L(\A)$.
So if $\L(\A)$ is non-empty, it is partition regular.
\end{proof}

Last, partition regular classes enjoy a very useful property: the intersection of an infinite decreasing sequence of partition regular classes is again partition regular. This property will be used to propage properties of $\Pi^0_2$ partition regular classes to arbitrary intersections of $\Sigma^0_1$ partition large classes.

\begin{proposition}\label{prop:interlarge}
Suppose $\{\L_n\}_{n \in \omega}$ is a collection of partition regular classes with $\L_{n+1} \subseteq \L_n$. Then $\bigcap_{n \in \omega} \L_n$ is partition regular.
\end{proposition}
\begin{proof}
For every $n$,  $\omega \in \L_n$ because $\L_n$ is partition regular. It follows that $\omega \in \bigcap_{n \in \omega} \L_n$. In particular $\bigcap_{n \in \omega} \L_n$ is not empty.

Suppose $X \in \bigcap_{n \in \omega} \L_n$. Let $Y \supseteq X$. For every $n$, since $X \in \L_n$ then $Y \in \L_n$ as $\L_n$ is partition regular. Thus $Y \in \bigcap_{n \in \omega} \L_n$.

Suppose $X \in \bigcap_{n \in \omega} \L_n$. Let $Y_0 \cup \dots \cup Y_k \supseteq X$. Suppose for contradiction that for every $i \leq k$ the set $Y_i$ is not in $\bigcap_{n \in \omega} \L_n$. For every $i \leq k$, let $n_i$ be such that $Y_i \notin \L_{n_i}$. Let $n$ be larger than these numbers. For every $i \leq k$, since $\L_{n_i} \supseteq \L_n$, the set $Y_i$ is not in $\L_n$. As $X \in \L_n$, it follows that $\L_n$ is not partition regular, which contradicts our hypothesis. Thus for every $X \in \bigcap_{n \in \omega} \L_n$ and for every $Y_0 \cup \dots \cup Y_k \supseteq X$, there is some $i \leq k$ such that $Y_i \in \bigcap_{n \in \omega} \L_n$.
\end{proof}

\subsection{$\Pi^0_2$ Partition regular classes}

The most basic non-trivial partition regular class, the class of all infinite sets, is $\Pi^0_2$. In this section, we study a few specific $\Pi^0_2$ partition regular classes and show that there is no non-trivial $\Sigma^0_2$ partition regular class.

\begin{proposition}
Let $\U$ be an upward-closed $\Sigma^0_1$ class. Then $\L(\U)$ is $\Pi^0_2$.
\end{proposition}
\begin{proof}
By \Cref{prop:complarge}, the largest partition regular subclass of $\U$ is defined by
$$\L(\U) = \{X \in 2^\omega\ :\ \forall k\ \forall X_0\cup\dots\cup X_k \supseteq X\ \exists i \leq k\ X_i \in \U\}$$
which is clearly $\Pi^0_2$.
\end{proof}

The previous proposition will be very useful for our computational analysis of partition regularity, as shows the following corollary.

\begin{corollary} \label{cor:complarge}
Let $\U$ be a $\Sigma^0_1$ class. The sentence ``$\U$ contains a partition regular class'' is $\Pi^0_2$.
\end{corollary}
\begin{proof}
By \cref{prop:complarge}, the class $\U$ contains a partition regular class iff $\omega \in \L(\U)$, which is a $\Pi^0_2$ sentence.
\end{proof}

The partition regular class of all infinite sets can be generalized to a whole family of non-trivial partition regular classes in the following way.

\begin{definition}
For any infinite set $X$ we define $\L_X$ as the $\Pi^0_2(X)$ partition regular class of the sets that intersect $X$ infinitely often.
\end{definition}

In particular, $\L_\omega$ is the class of all infinite sets. It is not the only possible kind of $\Pi^0_2$ partition regular class. There are examples of $\Pi^0_2$ partition regular classes $\L$ such that $\L_X \nsubseteq \L$ for any $X \in [\omega]^{\omega}$. Consider for instance the class $\{X \subseteq \omega :\ \sum_{n \in X} \frac{1}{n} = \infty\}$.

We finish this section by proving that there the only $\mathbf{\Sigma^0_2}$ partition regular classes are the trivial ones.

\begin{proposition}
There are no non-trivial $\mathbf{\Sigma^0_2}$ partition regular classes.
\end{proposition}
\begin{proof}
Fix a set~$Z$.
If a $\Sigma^{0,Z}_1$ class contains all the finite sets, then by Mileti~\cite{mileti2004partition}, it also contains all the $Z$-hyperimmune sets : Let $\U$ be such a $\Sigma^{0,Z}_1$ class. Let $X \not\in \U$.  For any $n$ and every string $\sigma$ of length $n$ we have $\sigma 0^{\infty} \in \U$. Let then $Z$-compute $f(n)$ such that for any $\sigma$ of length $n$ we have $[\sigma 0^{f(n)}] \subseteq \U$.
 It must be that $X\uh_{n+f(n)}$ contains at least a $1$ at a position greater than $n$. By repeating this we see that we can $Z$-computably bound $p_X$.

Suppose now for contradiction that there exists a non-trivial $\mathbf{\Sigma^0_2}$ partition regular class~$\C$. Let~$Z$ be such that~$\C$ is $\Sigma^{0,Z}_2$. In particular, $2^\omega \setminus \C$ contains all the hyperimmune sets. Consider now any set $X$ such that both $X$ and $\overline{X}$ are $Z$-hyperimmune. By partition regularity, at least one of them belongs to~$\C$, which is a contradiction.
\end{proof}

\subsection{Partition largeness}

As mentioned earlier, a partition regular class represents a notion of largeness. However, a class containing a partition regular class is not necessarily itself partition regular. These classes admit a nice characterization.

\begin{definition}
A \emph{partition large class} is a non-empty collection of sets $\A \subseteq 2^\omega$ such that
\begin{itemize}
	\item[(a)] If $X \in \A$ and $Y \supseteq X$, then $Y \in \A$
	\item[(b)] For every~$k$, if $Y_0 \cup \dots \cup Y_k \supseteq \omega$, there is some $j \leq k$ such that $Y_j \in \A$.
\end{itemize}
\end{definition}

\begin{proposition}\label[proposition]{prop:large-contains-pr}
A class $\A \subseteq 2^\omega$ is partition large iff it is upward-closed and contains a partition regular subclass.
\end{proposition}
\begin{proof}
Suppose $\A$ is upward-closed and contains a partition regular subclass $\L \subseteq \A$.
(a) is trivially satisfied by hypothesis. By partition regularity of $\L$, $\omega \in \L$ and for every~$k$ and every $Y_0 \cup \dots \cup Y_k \supseteq \omega$, there is some $j \leq k$ such that $Y_j \in \L \subseteq \A$. So $\A$ is partition large.

Suppose now $\A$ is partition large. By (a), it is upward-closed. We claim that $\L(\A)$ is partition regular. By (b) and \Cref{prop:complarge}, $\omega \in \L(\A)$, so $\L(\A) \neq \emptyset$. By definition of $\L(\A)$, it is partition regular, hence $\A$ contains a partition regular subclass.
\end{proof}

\begin{lemma} \label{lem:partition-regular-x}
Let $\L \subseteq 2^\omega$ be a non-trivial partition regular class and $X \in \L$. Then $\L \cap \L_X$ is partition large.
\end{lemma}
\begin{proof}
Let $Y_0 \cup \dots \cup Y_k \supseteq \omega$. In particular we have $Y_0 \cap X \cup \dots \cup Y_k \cap X \supseteq X$. As $X \in \L$ we must have $Y_j \cap X \in \L$ for some $j \leq k$. In particular, since~$\L$ is non-trivial, $Y_j \cap X$ is infinite, so $Y_j \cap X \in \L_X$. Therefore, there is some $j \leq k$ such that $Y_j \cap X \in \L \cap \L_X$.
\end{proof}

The following proposition yields another definition of $\L(\A)$ which will be very useful.

\begin{proposition}\label[proposition]{lem:lx}
Let $\A \subseteq \U_2$ be a partition large class. Then
$$
\L(\A) = \{ X : \A \cap \L_X \mbox{ is partition large}\}
$$
\end{proposition}
\begin{proof}
Let $\L = \{ X : \A \cap \L_X \mbox{ is partition large} \}$.
Since $\A$ is partition large, then by \Cref{prop:large-contains-pr}, $\L(\A)$ is the largest partition regular subclass of $\A$. Moreover, by \Cref{prop:non-trivial-u2}, $\L(\A)$ is non-trivial since~$\L(\A) \subseteq \U_2$. Thus, by \Cref{lem:partition-regular-x}, for every $X \in \L(\A)$, $\L(\A) \cap \L_X$ is partition large. In particular, $\A \cap \L_X$ is partition large, hence $X \in \L$. It follows that $\L(\A) \subseteq \L$.

Let us show that $\L$ is partition regular. First, $\A \cap \L_\omega = \A$ is partition large, so $\omega \in \L$.
Let $X \in \L$ and $Y_0 \cup \dots \cup Y_k \supseteq X$. In particular, $\A \cap \L_X$ is partition large, so by \Cref{prop:large-contains-pr}, $\L(\A \cap \L_X)$ is partition regular. Moreover, $X \in \L(\A \cap \L_X)$, so there is some $i \leq k$ such that $Y_i \in \L(\A \cap \L_X)$. By \Cref{lem:partition-regular-x}, $\L(\A \cap \L_X) \cap \L_{Y_i}$ is partition large. Since $\L(\A \cap \L_X) \cap \L_{Y_i} \subseteq \A \cap \L_{Y_i}$, then $\A \cap \L_{Y_i}$ is partition large, so $Y_i \in \L$.

Last, let us show that $\L \subseteq \A$. Indeed, if $\A \cap \L_X$ is partition large, then since $X \cup \overline{X} = \omega$, either $X$, or $\overline{X}$ belongs to $\A \cap \L_X$. However, $\overline{X} \not \in \L_X$, so $X \in \A \cap \L_X$. In particular, $X \in \A$.
It follows that $\L$ is a partition regular subclass of $\A$, so by maximality of $\L(\A)$, $\L \subseteq \L(\A)$.
\end{proof}

Recall that \Cref{prop:non-trivial-u2} characterizes non-trivial partition regular classes as those which are included in the $\Sigma^0_1$ class $\U_2$ of sets containing at least two distinct elements. Since there is no non-empty $\Sigma^0_1$ class containing only infinite sets, we take this characterization as the natural generalization of non-triviality to arbitrary classes.

\begin{definition}
A class $\A \subseteq 2^\omega$ is \emph{non-trivial} if it is included in the $\Sigma^0_1$ class $\U_2$ of sets containing at least two distinct elements.
\end{definition}

\subsection{Partition genericity}\label[section]{subsect:pg}

Given a $\Pi^0_2$ partition regular class $\L$ and a set $X$, then either $X \in \L$, or $\overline{X} \in \L$. In general, whether the first or the second case holds depends on the choice of $\L$. For some sets however, the same case always holds. This yields the notion of partition genericity.


\begin{definition}
Let $\A \subseteq 2^\omega$ be a class. We say that $X$ is \emph{partition generic in $\A$} if $X$ belongs to every non-trivial $\Pi^0_2$ partition regular subclass of $\A$. If $X$ is partition generic in $2^\omega$ we simply say that $X$ is \emph{partition generic}.
\end{definition}

%

The first and most trivial example of partition generic set is $\omega$.
First, note that partition genericity is closed under finite changes.

\begin{proposition}\label[proposition]{prop:pg-invariant-finite}
If $X$ is partition generic in $\A$	and $Y =^{*} X$, then $Y$ is partition generic in~$\A$.
\end{proposition}
\begin{proof}
Let $\L \subseteq \A$ be any non-trivial $\Pi^0_2$ partition regular subclass of $\A$.
Since $X$ is partition generic in $\A$, then $X \in \L$. Since $Y =^{*} X$, by \Cref{prop:measure:pr}, $Y \in \L$.
Therefore $Y$ is partition generic in~$\A$.
\end{proof}

It follows that every co-finite set is partition generic. Actually, this characterizes the computable partition generic sets. Indeed, if $A$ is a co-infinite computable set, then $\L_{\overline{A}}$ is a non-trivial $\Pi^0_2$ partition regular class which does not contain~$A$.

\begin{definition}
Let $\A \subseteq 2^\omega$ be a class. We say that $X$ is \emph{bi-partition generic in $\A$} if $X$ and $\overline{X}$ are both partition generic in $\A$. If $X$ is bi-partition generic in $2^\omega$ we simply say that $X$ is \emph{bi-partition generic}.
\end{definition}

The existence of bi-partition generic sets follows from \Cref{prop:measure:pr}. Note that, contrary to partition genericity, no computable set is bi-partition generic.

\begin{proposition}\label[proposition]{prop:bi-partition-generic-biimmune}
Every bi-partition generic set is bi-immune.
\end{proposition}
\begin{proof}
Let~$A$ be bi-partition generic. Let~$X$ be an infinite subset of~$A$.
Then $\L_X$ is a partition regular class such that $\overline{A} \not \in \L_X$.
Since~$\overline{A}$ is partition generic, $\L_X$ is not $\Pi^0_2$, so $X$ is not computable.
Similarly, $\overline{A}$ has no infinite computable subset.
\end{proof}

In particular, every bi-partition generic set is bi-infinite. We now prove that every typical set is bi-partition generic, that is, every sufficiently random or generic set is bi-partition generic.

\begin{proposition}\label[proposition]{prop:kurtz-bipg}
Every Kurtz random is bi-partition generic.
\end{proposition}
\begin{proof}
By~\Cref{prop:measure:pr}, every non-trivial measurable partition regular class is of measure~$1$. It follows that any Kurtz-random belongs to every $\Pi^0_2$ partition regular class and thus that any Kurtz-random is bi-partition generic.
\end{proof}


In particular, one can be bi-hyperimmune and bi-partition generic.
The following proposition shows that is actually always the case,
in the sense that every bi-hyperimmune set is bi-partition generic.

\begin{proposition}\label[proposition]{prop:cohyp-pg}
Every co-hyperimmune set is partition generic.
\end{proposition}
\begin{proof}
Let $A$ be a co-hyperimmune set.
Suppose for the contradiction that $A \not \in \L$
for some non-trivial $\Pi^0_2$ partition regular class $\L \subseteq 2^\omega$. In particular, there is a partition large $\Sigma^0_1$ class $\U \supseteq \L$ such that $A \not \in \U$.
Since $\U$ is partition large, for every $t \in \omega$, there is some $\rho \in 2^{<\omega}$ with $\min \rho > t$ such that $\rho \in \U$.
In particular, for every such $\rho$, we have $\rho \cap \overline{A} \neq \emptyset$. Moreover, such a string~$\rho$ can be found computably uniformly in $t$. We can therefore compute an array tracing $\overline{A}$, contradicting hyperimmunity of $\overline{A}$.
\end{proof}

\begin{corollary}
Every bi-hyperimmune set is bi-partition generic.
\end{corollary}
\begin{proof}
Immediate by~\Cref{prop:cohyp-pg}.
\end{proof}

As mentioned earlier, every sufficiently random set is bi-partition generic. Moreover, every sufficiently random set is effectively bi-immune. It is natural to wonder whether every effectively co-immune set is partition generic. The following proposition answers negatively.

\begin{proposition}\label[proposition]{prop:co-immune-not-pg}
There is an effectively co-immune set which is not partition generic.
\end{proposition}
\begin{proof}
Consider the following $\Pi^0_2$ class $\L = \{X\ :\ \forall k\ \exists n\ |X \uh_{n^2}| \geq nk \}$. Let us prove it is partition regular : Let $X \in \L$ and $Y_1 \cup \dots Y_m \supseteq X$. For $k \in \omega$ there is some~$n$ such that $|X \uh_{n^2}| \geq n m k$. Then, there is some $e_k \leq m$ such that $|Y_{e_k} \uh_{n^2}| \geq n k$. Let $e \leq m$ be such that $e = e_k$ for infinitely many $k$. We then have for infinitely many $k$ that there exists $n$ such that $|Y_{e} \uh_{n^2}| \geq n k$. This is therefore true in particular for every $k$. Thus $Y_e \in \L$, so $\L$ is partition regular.

It suffices to construct an effectively co-immune set $A$ such that for every $n$, $|A \uh_{n^2}| < n$ for every $n$. Then $A \not \in \L$, hence $A$ is not partition generic.
To construct such a set, for every $e$, let $x_e$ be the $e^2$th element of $W_e$ (in the $<_\omega$ order), if it exists. Assume by convention that $W_0 = \emptyset$, so $e_0$ does not exist. Let $A = \{ x_e : e \in \omega \}$. Then $A$ is an infinite set, which is effectively co-immune, as witnessed by the fonction $n \mapsto n^2$. Last, for every $n$, $A \uh_{n^2} \subseteq \{ x_e : e < n \}$, so $|A \uh_{n^2}| < n$.
\end{proof}

We have seen so far three classes of partition generic sets: co-finite sets, co-hyperimmune sets and Kurtz randoms. Let us construct a bi-partition generic set which belongs to none of these categories.

\begin{proposition}\label[proposition]{prop:bipartition-generic-2n}
There is a bi-partition generic set~$A$ such that for every~$n$, $A(2n) \neq A(2n+1)$.
\end{proposition}
\begin{proof}
Consider the notion of forcing whose conditions are strings $\sigma$ of even length, such that for every~$n < |\sigma|/2$, $\sigma(2n) \neq \sigma(2n+1)$. The conditions are partially ordered by the suffix relation. Let us show that every sufficiently generic set~$G$ is bi-partition generic.

Let~$\sigma$ be a condition, and let $\U$ be a non-trivial $\Sigma^0_1$ partition large class. Let $X_0 = \{2n : 2n > |\sigma| \}$ and $X_1 = \{2n+1 : 2n+1 > |\sigma|$. Since $\{0\} \cup \dots \cup \{|\sigma|\} \cup X_0 \cup X_1 \supseteq \omega$, then either $X_0 \in \U$, or $X_1 \in \U$. Say the former case holds as the other case is symmetric. Since $\U$ is $\Sigma^0_1$, then there is some $k$ such that $[X_0 \uh_{2k}] \subseteq \U$. Let $\tau = \sigma \cup X_0 \uh_{2k}$. Note that $\tau$ is a valid condition. Moreover, since $\U$ is closed under superset, then $[\tau] \subseteq \U$. Therefore, every sufficiently generic set~$G$ for this notion of forcing belongs to every non-trivial $\Sigma^0_1$ partition large class. By symmetric, so does the complement of~$G$, so $G$ is bi-partition generic.
\end{proof}

\begin{corollary}
There is a bi-partition generic set which is neither co-hyperimmune, nor Kurtz random.
\end{corollary}
\begin{proof}
Consider the set~$A$ of \Cref{prop:bipartition-generic-2n}.
It is clearly neither hyperimmune nor co-hyperimmune since~$F_0, F_1, \dots$ defined by~$F_n = \{2n, 2n+1\}$ is a c.e. array tracing both~$A$ and $\overline{A}$.  
Furthermore, $A$ is not Kurtz random, since~$A \in \{ X : \forall n X(2n) \neq X(2n+1) \}$ which is a $\Pi^0_1$ class of measure~0.
\end{proof}


The following lemma is a sort of pigeonhole principle, from which we will derive partition regularity of the class of sets which are partition generic in some non-trivial $\Sigma^0_1$ partition large class.

\begin{lemma} \label{lemma:abslarge2}
Let $\A \subseteq 2^\omega$ be a class and $X$ be a set which is partition generic in $\A$. For every $Y_0 \cup Y_1 \supseteq X$, if $Y_0 \notin \L(\A)$ then $Y_1$ is partition generic in $\A$.
\end{lemma}
\begin{proof}
Suppose for contradiction that there is a non-trivial $\Pi^0_2$ partition regular class $\V \subseteq \A$ such that $Y_1 \notin \V$. In particular, $\V \subseteq \L(A)$, so $Y_0 \notin \V$. By partition regularity of $\V$, since $Y_0 \cup Y_1 \supseteq X$, then $X \notin \V$, which contradicts partition genericity of $X$ in $\A$.
\end{proof}

\begin{proposition} \label[proposition]{lemma:abslarge}
Let $\A$ be a partition large class. Suppose $X$ is partition generic in $\A$. Let $Y_0 \cup \dots \cup Y_k \supseteq X$. Then there is a $\Sigma^0_1$ class $\U$ such that $\U \cap \A$ is partition large, together with some $i \leq k$ such that $Y_i$ is partition generic in $\U \cap \A$.
\end{proposition}
\begin{proof}
We prove the statement by the lemma by induction on~$k$.
For $k = 0$, $Y_0 \supseteq X$ so $Y_0$ is partition generic in~$\A$ by upward-closure of partition genericity. Take $\U = 2^\omega$ and we are done.

Suppose now that the property holds for $k-1$. Suppose $Y_k$ is not partition generic in $\A$.
Thus there is a non-trivial $\Pi^0_2$ partition regular class $\bigcap_{e \in C} \U_e \subseteq \A$ such that $Y_k \not \in \bigcap_{e \in C} \U_e$. In particular, there is some $e \in \C$ such that $Y_k \notin \U_e$. Note that $\bigcap_{e \in C} \U_e \subseteq \A \cap \U_e$, so $\A \cap \U_e$ is partition large. By \Cref{lemma:abslarge2}, $Y_0 \cup \dots \cup Y_{k-1}$ is partition generic in $\U_e \cap \A$. By induction hypothesis on $\U_e \cap \A$ and $Y_0 \cup \dots \cup Y_{k-1}$, there is a $\Sigma^0_1$ class $\V$ such that $\V \cap \U_e \cap \A$ is partition large, together with some $i < k$ such that $Y_i$ is partition generic in $\V \cap \U_e \cap \A$. The property therefore holds with the $\Sigma^0_1$ class $\V \cap \U_e$.
\end{proof}

\begin{corollary}\label[corollary]{cor:partition-generic-regular}
The following class is partition regular : 
$$
\P = \{ X : X \mbox{ is partition generic in some non-trivial } \Sigma^0_1 \mbox{ partition large class} \}
$$
\end{corollary}
\begin{proof}
First, $\omega$ is partition generic in $\U_2 = \{ X : |X| \geq 2 \}$.
Suppose $X \in \P$ and let $Y_0 \cup \dots \cup Y_k \supseteq X$. Let $\U \subseteq 2^\omega$ be a non-trivial $\Sigma^0_1$ partition large class in which $X$ is partition generic. By \Cref{lemma:abslarge}, there is a $\Sigma^0_1$ class $\V$ such that $\U \cap \V$ is partition large, together with some $i \leq k$ such that $Y_i$ is partition generic in $\U \cap \V$. Note that $\U \cap \V \subseteq \U$, hence is non-trivial. It follows that $Y_i \in \P$.
\end{proof}

Partition genericity within a large class is not an interesting notion of typicality whenever one does not put effectiveness restrictions on the large classes. Indeed, we shall see that every infinite set is partition generic relative to itself within a large class.
Given an infinite set~$X$, let $\U_{X,2} = \{ Y : |Y \cap X| \geq 2\}$.

\begin{lemma}\label[lemma]{lem:ua2-non-trivial}
For every set~$X$, $\L(\U_{X,2}) = \L_X$.
\end{lemma}
\begin{proof}
First, let us show that $\L_X \subseteq \L(\U_{X,2})$.
Suppose that~$Y \in \L_X$. By definition, $|Y \cap X| = \infty$, so~$Y \in \U_{X,2}$.
So~$\L_X$ is a partition regular subclass of~$\U_{X,2}$. Since $\L(\U_{X,2})$ is the largest partition regular subclass of~$\U_{X,2}$, then $\L_X \subseteq \L(\U_{X,2})$.

Then, let us show that $\L(\U_{X,2}) \subseteq \L_X$.
Let~$Y \in \L(\U_{X,2})$. We claim that~$|Y \cap X| = \infty$. Indeed, otherwise, consider the $|Y \cap X|$-cover of $Y \cap X$ made of singletons. By partition regularity of $\L(\U_{X,2})$, one of the parts belongs to~$\U_{X,2}$, but $\U_{X,2}$ contains no singleton, contradiction. Therefore~$|Y \cap X| = \infty$, so $Y \in \L_X$.
\end{proof}

\begin{lemma}\label[lemma]{lem:ua2-large}
For every infinite set~$X$, $\U_{X,2}$ is a non-trivial $\Sigma^0_1(X)$ partition large class.
\end{lemma}
\begin{proof}
$\U_{X,2}$ is clearly non-trivial, upward-closed and $\Sigma^0_1(X)$. 
By \Cref{lem:ua2-non-trivial}, $\L(\U_{X,2}) = \L_X$. Since~$X$ is infinite, $X \in \L_X$, so $\L(\U_{X,2}) \neq \emptyset$. By \Cref{prop:large-contains-pr}, $\U_{X,2}$ is large. 
\end{proof}

The following lemma shows, as promised, that every infinite set is partition generic relative to itself within a large class.

\begin{lemma}\label[lemma]{lem:ua2-pg}
For every set~$X$, $X$ belongs to any large subclass of $\U_{X,2}$.
\end{lemma}
\begin{proof}
Let~$\V$ be a large subclass of~$\U_{X,2}$.
Suppose for the contradiction that~$X \not \in \V$. By largeness of $\V$, $\overline{X} \in \V \subseteq \U_{X,2}$. Then $|\overline{X} \cap X| \geq 2$, contradiction.
\end{proof}

By Patey~\cite{patey2017iterative}, if $f$ is a hyperimmune function, then for every set $A$, there is an infinite subset $H \subseteq A$ or $H \subseteq \overline{A}$ such that $f$ is $H$-hyperimmune. Is it also the case if one replaces hyperimmunity by partition genericity? We answer negatively by constructing a specific bi-partition generic set which is neither Kurtz random, nor bi-hyperimmune.

\begin{lemma}\label[lemma]{lem:subset-complement-not-pg}
For every set~$A$ and every infinite set~$H \subseteq \overline{A}$, then~$A$ is not $H$-partition generic.
\end{lemma}
\begin{proof}
By \Cref{lem:ua2-large}, $\U_{H,2}$ is a non-trivial $\Sigma^0_1(A)$ partition large class. In particular, $\L(\U_{H,2})$ is a non-trivial $\Pi^0_2(H)$ partition generic class.
However, $A \not \in \U_{H,2} \supseteq \L(\U_{H,2})$, so~$A$ is not $H$-partition generic.
\end{proof}

\begin{proposition}
There is a bi-partition generic set~$A$ such that for every infinite set~$H \subseteq A$ and~$H \subseteq \overline{A}$, neither~$A$ nor~$\overline{A}$ is $H$-partition generic.
\end{proposition}
\begin{proof}
Let $A$ be the bi-partition generic set of~\Cref{prop:bipartition-generic-2n}.
Let~$H \subseteq A$. By \Cref{lem:subset-complement-not-pg}, $\overline{A}$ is not $H$-partition generic. Let~$P = \{ 2n+1 : 2n \in H \} \cup \{2n : 2n+1 \in H \}$. Then~$P \subseteq \overline{A}$. By \Cref{lem:subset-complement-not-pg}, $\overline{A}$ is not $P$-partition generic, hence not~$H$-partition generic.
The case $H \subseteq \overline{A}$ is symmetric.
\end{proof}

If we consider partition genericity in a non-trivial partition large open class, then the answer is positive, but in unsatisfactory manner.

\begin{proposition}
Let $B$ be an infinite set. For every set $A$, there is an infinite subset $H \subseteq A$ or $H \subseteq \overline{A}$ such that $B$ is $H$-partition generic in a non-trivial $\Sigma^0_1(H)$ partition large class.
\end{proposition}
\begin{proof}
Let $H = A \cap B$ if it is infinite, otherwise $H = \overline{A} \cap B$.
By \Cref{lem:ua2-large}, $\U_{H,2}$ is a non-trivial  $\Sigma^0_1(H)$ partition large class.
We claim that $B$ is $H$-partition generic in~$\U_{H,2}$.
By \Cref{lem:ua2-pg}, $H$ is $H$-partition generic in~$\U_{H,2}$.
Since~$B \supseteq H$, then $A$ is also $H$-partition generic in~$\U_{H,2}$.
\end{proof}


\section{Applications}\label[section]{sect:applications}

We now justify the study of partition genericity by proving several partition genericity subset basis theorems. 
All these basis theorems are proven with the same notion of forcing, that we call \emph{partition generic Mathias forcing}. A \emph{condition} is a tuple $(\sigma, X, \U)$, where $(\sigma, X)$ is a Mathias condition, that is, $\sigma$ is a finite string and $X$ is an infinite set such that $\min X > |\sigma|$. Moreover, $\U$ is a non-trivial large $\Sigma^0_1$ class within which $X$ is partition generic. A condition $(\tau, Y, \V)$ \emph{extends} another condition $(\sigma, X, \U)$ if $(\tau, Y)$ Mathias extends $(\sigma, X)$, that is, $\sigma \prec \tau$, $Y \subseteq X$ and $\tau \setminus \sigma \subseteq X$. Moreover, we require that $\V \subseteq \U$.

Note that we do not impose any effectivity restriction to the reservoir $X$. In particular, if $A$ is a set which is partition generic in a non-trivial $\Sigma^0_1$ large class $\U$, then we will consider sufficiently generic filters containing the condition $(\emptyset, A, \U)$. Indeed, any such filter $\F$ induces a subset $G_\F = \bigcup_{(\sigma, X, \V) \in \F} \sigma$ of~$A$. 

The first property we prove is that $G_\F$ is an infinite set. 

\begin{lemma}\label[lemma]{lem:pg-forcing-infinite}
Let $\F$ be a sufficiently generic filter. Then $G_\F$ is infinite.
\end{lemma}
\begin{proof}
Let us show that for any $n$, the set of conditions $(\tau, Y, \U)$ such that $\tau(x) = 1$ for some $x > n$ is dense.
Let $(\sigma, X, \U)$ be a condition. Since $X$ is partition generic in $\U$ and $\U$ is non-trivial, then $X \in \L(\U)$, hence $X$ is infinite. Let $x \in X$ be greater than $n$, let $\tau$ be string $\sigma \cup \{x\}$ and let $Y = X \setminus \{0, \dots, |\rho|\}$. By \Cref{prop:pg-invariant-finite}, $Y$ is partition generic in $\U$. Therefore, $(\tau, Y, \U)$ is a valid extension of $(\sigma, X, \U)$.
\end{proof}

We now turn to the various partition genericity subset basis theorems. 

\subsection{Cone avoidance}

The first one is the \emph{cone avoidance partition genericity subset basis theorem}. Recall that Dzhafarov and Jockusch~\cite{Dzhafarov2009Ramseys} proved that if $B$ is a non-computable set, then for every set $A \subseteq \omega$, there is an infinite set $H \subseteq A$ or $H \subseteq \overline{A}$ such that $B \not \leq_T H$. We prove the corresponding partition genericity theorem.

\begin{lemma}\label[lemma]{lem:pg-forcing-cone-avoidance}
Let $B$ be a non-computable set and let $\Phi_e$ be a Turing functional.
Then for every condition $(\sigma, X, \U)$, there is an extension $(\tau, Y, \U)$
forcing $\Phi_e^G \neq B$.	
\end{lemma}
\begin{proof}
A \emph{split pair} is a pair of strings $\rho_0, \rho_1$ such that there is an input~$x \in \omega$ for which $\Phi_e^{\sigma \cup \rho_0}(x)\downarrow \neq \Phi_e^{\sigma \cup \rho_1}(x)\downarrow$. We have two cases.

Case 1: the following class is large:
$$
\A = \{ Y \in \U : Y \emph{ contains a split pair } \}
$$
Since $\A$ is a non-trivial $\Pi^0_2$ large subclass of~$\U$ and $X$ is partition generic in~$\U$, then $X \in \A$. Let $\rho_0, \rho_1 \subseteq X$ be a split pair, with witness~$x$.
In particular, there is some $i < 2$ such that $\Phi_e^{\sigma \cup \rho_i}(x)\neq B(x)$.
Let $Y = X \setminus \{0, \dots, |\rho_i|\}$. By \Cref{prop:pg-invariant-finite}, $Y$ is partition generic in~$\U$.
Then $(\sigma \cup \rho_i, Y, \U)$ is an extension of $(\sigma, X, \U)$ forcing $\Phi_e^G(x) \neq B(x)$.

Case 2: there is some $k \in \omega$ such that the following class is non-empty:
$$
\C = \{ Z_0 \oplus \dots \oplus Z_{k-1} : Z_0 \cup \dots \cup Z_{k-1} = \omega \wedge
 \forall i < k\ Z_i\, \not \in \U \mbox{ or } Z_i \mbox{ contains no split pair } \}
$$
Note that $\C$ is a $\Pi^0_1$ class. By the cone avoidance basis theorem~\cite{Jockusch197201}, there is an element $Z_0 \oplus \dots \oplus Z_{k-1} \in \C$ such that $B \not \leq_T Z_0 \oplus \dots \oplus Z_{k-1}$. By \Cref{lemma:abslarge}, there is a $\Sigma^0_1$ class $\V$ such that $\V \cap \U$ is large and some $i < k$ such that $Z_i \cap X$ is partition generic in $\V  \cap \U$. Note that in particular, $Z_i \cap X \in \V  \cap \U$, so $Z_i \in \U$, hence~$Z_i$ contains no split pair.

The condition $(\sigma, Z_i \cap X, \V \cap \U)$ is a valid extension of $(\sigma, X, \U)$. 

We claim that $(\sigma, Z_i \cap X, \V \cap \U)$ forces $\Phi^G_e \neq B$.
Suppose for the contradiction that $\Phi^{G_\F}_e = B$ for some generic filter~$\F$.
Let $f : \omega \to 2$ be the partial $Z_i$-computable function which on input~$x$,
searches for some $\rho \subseteq Z_i$ such that $\Phi^{\sigma \cup \rho}_e(x)\downarrow$. If such a~$\rho$ is found, then $f(x) = \Phi^{\sigma \cup \rho}_e(x)\downarrow$.
Since $\Phi^{G_\F}_e = B$, then $f$ is total. Since $Z_i$ contains no split pair, then $f = B$, which contradicts the assumption that $B \not \leq_T Z_i$.
\end{proof}

\begin{theorem}\label[theorem]{thm:pg-basis-cone-avoidance}
Let $A$ be a set which is partition generic in a non-trivial $\Sigma^0_1$ large class $\U \subseteq 2^\omega$. Then for every non-computable set $B$, there is an infinite set $G \subseteq A$ such that $B \not \leq_T G$.
\end{theorem}
\begin{proof}
Let $\F$ be a sufficiently generic filter containing $(\emptyset, A, \U)$.
By \Cref{lem:pg-forcing-infinite}, $G_\F$ is infinite. By construction, $G_\F \subseteq A$.
Last, by \Cref{lem:pg-forcing-cone-avoidance}, $B \not \leq_T G_\F$.
\end{proof}

\begin{corollary}[Dzhafarov and Jockusch~\cite{Dzhafarov2009Ramseys}]
Let $B$ be a non-computable set. Then for every set $A \subseteq \omega$, there is an infinite set $H \subseteq A$ or $H \subseteq \overline{A}$ such that $B \not \leq_T H$.
\end{corollary}
\begin{proof}
By \Cref{cor:partition-generic-regular}, either $A$ or $\overline{A}$ is partition generic in a non-trivial $\Sigma^0_1$ large class. Apply \Cref{thm:pg-basis-cone-avoidance}.
\end{proof}

\subsection{PA avoidance}

We now turn to a PA avoidance partition genericity subset basis theorem.
Liu~\cite{Liu2012RT22} proved that for every set~$A \subseteq \omega$, there is an infinite set~$H \subseteq A$ or $H \subseteq \overline{A}$ of non-PA degree. We prove the corresponding partition genericity theorem.

\begin{lemma}\label[lemma]{lem:pg-forcing-non-pa}
For every condition $(\sigma, X, \U)$ and every Turing functional $\Phi_e$,
there is an extension $(\tau, Y, \V)$ forcing $\Phi^G_e$ not to be a DNC${}_2$ function.	
\end{lemma}
\begin{proof}
For every $k$, let $S_k$ be the set of pairs $(x, v) \in \omega \times 2$ such that for every $k$-cover $Z_0 \sqcup \dots \sqcup Z_{k-1} \supseteq \omega$, there is some $j < k$ such that $Z_j \in \U$, and some $\rho \subseteq Z_j$ such that $\Phi_e^{\sigma \cup \rho}(x)\downarrow = v$.
Note that the set $S_k$ is $\Sigma^0_1$ uniformly in $k$. We have two cases:

Case 1: for every $k \in \omega$, there is some $x \in \omega$ such that $(x, \Phi_x(x)) \in S_k$. Then the following class is large:
$$
\{ Y \in \U : \exists x\ \exists \rho \subseteq Y\  \Phi^{\sigma \cup \rho}_e(x)\downarrow = \Phi_x(x) \}
$$
Since $X$ is partition generic, then it belongs to this class, so there is some $x$ and some $\rho \subseteq X$ such that $\Phi^{\sigma \cup \rho}_e(x)\downarrow = \Phi_x(x)$. Let $Y = X \setminus \{0, \dots, |\rho|\}$. Note that $Y$ is again partition generic in $\U$, so $(\sigma \cup \rho, Y, \U)$ is an extension forcing $\Phi^G_e(x)\downarrow = \Phi_x(x)$.

Case 2: there is some $k \in \omega$ such that for every $x \in \omega$, if $\Phi_x(x)\downarrow$, then $(x, \Phi_x(x)) \not \in S_k$. Then there must be some $x$ such that $(x, 0), (x, 1) \not \in S_k$, otherwise we would compute a DNC${}_2$ function by waiting, for each $x$, for some $v < 2$ such that $(x, v)$ is enumerated in $S_k$.
   Let $x$ be such that $(x, 0), (x, 1) \not \in S_k$. By definition of $S_k$, there are two $k$-covers of $\omega$, $Z^0_0 \cup \dots \cup Z^0_{k-1} = \omega$ and  $Z^1_0 \cup \dots \cup Z^1_{k-1} = \omega$ such that for every $i < 2$ and every $j < k$ such that $Z^i_j \in \U$, and every $\rho \subseteq Z^i_j$ such that $\Phi_e^{\sigma \cup \rho}(x)\downarrow$, then $\Phi_e^{\sigma \cup \rho}(x) \neq i$. 
   
 Then $\langle Z^0_s \cap Z^1_t : s, t < k \rangle$ is a $k^2$-cover of $\omega$ such that for every $s, t < k$, either $Z^0_s \cap Z^1_t \not \in \U$, or for every $\rho \subseteq Z^i_j$, $\Phi_e^{\sigma \cup \rho}(x)\uparrow$. By \Cref{lemma:abslarge}, there is a non-trivial large $\Sigma^0_1$ class $\V \subseteq \U$ and some $s, t < k$ such that $Z^0_s \cap Z^1_t \cap X$ is partition regular in $\V$. The condition $(\sigma, Z^0_s \cap Z^1_t \cap X, \V)$ is an extension forcing $\Phi^G_e(x)\uparrow$.
 \end{proof}

\begin{theorem}\label[theorem]{thm:theorem-partition-generic-liu}
Let $A$ be a set which is partition generic in a non-trivial large $\Sigma^0_1$ class $\U \subseteq 2^\omega$. Then there is an infinite set $G \subseteq A$ of non-PA degree.
\end{theorem}
\begin{proof}
Let $\F$ be a sufficiently generic filter containing $(\emptyset, A, \U)$.
By \Cref{lem:pg-forcing-infinite}, $G_\F$ is infinite. By construction, $G_\F \subseteq A$.
Last, by \Cref{lem:pg-forcing-non-pa}, $G_\F$ is of non-PA degree.
\end{proof}

\begin{corollary}[Liu~\cite{Liu2012RT22}]
For every set $A \subseteq \omega$, there is an infinite subset $H \subseteq A$ or $H \subseteq \overline{A}$ of non-PA degree.
\end{corollary}
\begin{proof}
By \Cref{cor:partition-generic-regular}, either $A$ or $\overline{A}$ is partition generic in a non-trivial $\Sigma^0_1$ large class. Apply \Cref{thm:theorem-partition-generic-liu}.
\end{proof}

\begin{remark}
In all the arguments above, one could have worked with a generalized notion of condition
$(\sigma, X, \L)$ where $\L$ is a non-trivial $\Pi^0_2$ partition regular class.
Then, for every sufficiently generic filter $\F$ and every condition $(\sigma, X, \L)$,
$G_\F \in \L$. This implies in particular that all these partition generic basis theorems also hold if one replaces \qt{there exists an infinite subset $H \subseteq A$} by \qt{there exists a subset $H \subseteq A$ such that $H \in \L$} for any non-trivial $\Pi^0_2$ partition regular class~$\L$.
\end{remark}

\subsection{Constant-bound trace avoidance}

Soon after proving his PA avoidance pigeonhole basis theorem, Liu proved a constant bound trace avoidance pigeonhole basis theorem, with several consequences, such as the existence, for every set~$A \subseteq \omega$, of an infinite subset~$H \subseteq A$ or $H \subseteq \overline{A}$ which does not compute any Martin-L\"of random, or even no set of positive Hausdorff dimension.

\begin{definition}
A \emph{$k$-trace} is a sequence of finite sets of strings $F_0, F_1, \dots$ such that for every $n$, $|F_n| = k$ and every string $\sigma \in F_n$ is of length~$n$. A \emph{constant-bound trace} is a $k$-trace for some $k \in \omega$.
A $k$-trace of a class $\C \subseteq 2^\omega$ is a $k$-trace $F_0, F_1, \dots$ such that $[F_n] \cap \C \neq \emptyset$ for every~$n$, where $[F_n] = \bigcup_{\sigma \in F_n} [\sigma]$.
\end{definition}

Liu~\cite{liu2015cone} proved that if $\C \subseteq 2^\omega$ is a non-empty $\Pi^0_1$ class with no computable 1-trace, then for every set~$A \subseteq \omega$, there is an infinite subset~$H \subseteq A$ or $H \subseteq \overline{A}$ which does not compute any 1-trace of~$\C$. We prove the corresponding partition genericity subset basis theorem.

\begin{lemma}\label[lemma]{lem:pg-forcing-non-trace}
Let $\C \subseteq 2^\omega$ be a non-empty $\Pi^0_1$ class with no computable 1-trace. For every condition $(\sigma, X, \U)$ and every Turing functional $\Phi_e$,
there is an extension $(\tau, Y, \V)$ forcing $\Phi^G_e$ not to be a 1-trace of~$\C$.
\end{lemma}
\begin{proof}
For every $k$, let $S_k$ be the set of strings $\mu \in 2^{<\omega}$ such that for every $k$-cover $Z_0 \sqcup \dots \sqcup Z_{k-1} \supseteq \omega$, there is some $j < k$ such that $Z_j \in \U$, and some $\rho \subseteq Z_j$ such that $\Phi_e^{\sigma \cup \rho}(|\mu|)\downarrow = \mu$.
Note that the set $S_k$ is $\Sigma^0_1$ uniformly in $k$. We have two cases: 

Case 1: for every $k \in \omega$, there is some $\mu \in S_k$ such that $[\mu] \cap \C = \emptyset$. Then the following class is large:
$$
\{ Y \in \U : \exists \mu\ \exists \rho \subseteq Y\  \Phi^{\sigma \cup \rho}_e(|\mu|)\downarrow = \mu \wedge [\mu] \cap \C = \emptyset \}
$$
is large. Since $X$ is partition generic, then it belongs to this class, so there is some $\mu \in 2^{<\omega}$ and some $\rho \subseteq X$ such that $\Phi^{\sigma \cup \rho}_e(|\mu|)\downarrow = \mu$ and $[\mu] \cap \C = \emptyset$. Let $Y = X \setminus \{0, \dots, |\rho|\}$. Note that $Y$ is again partition generic in $\U$, so $(\sigma \cup \rho, Y, \U)$ is an extension forcing $\Phi^G_e(|\mu|)\downarrow = \mu$.

Case 2: there is some $k \in \omega$ such that for every $\mu \in S_k$, $[\mu] \cap \C \neq \emptyset$. Then there must be some $n$ such that $S_k$ contains no string of length~$n$. Indeed, otherwise, one would compute a 1-trace of~$\C$, contradicting our hypothesis. 
    For every string~$\mu$ of length~$n$, let $Z^\mu_0 \cup \dots \cup Z^\mu_{k-1} = \omega$ be a $k$-cover such that for every~$j < k$ such that $Z^\mu_j \in \U$, and every $\rho \subseteq Z^\mu_j$ such that $\Phi_e^{\sigma \cup \rho}(n)\downarrow$, then $\Phi_e^{\sigma \cup \rho}(n) \neq \mu$. 
   
 Let $\langle P_j : j < k^{2^n} \rangle$ be the $k^{2^n}$-cover of~$\omega$ refining all the $k$-covers above. Then for every $j < k^{2^n}$ and every $\rho \subseteq P_j$, $\Phi_e^{\sigma \cup \rho}(n)\uparrow$. By \Cref{lemma:abslarge}, there is a non-trivial large $\Sigma^0_1$ class $\V \subseteq \U$ and some $j < k^{2^n}$ such that $P_j \cap X$ is partition regular in $\V$. The condition $(\sigma, P_j \cap X, \V)$ is an extension forcing $\Phi^G_e(n)\uparrow$.
 \end{proof}

\begin{theorem}\label[theorem]{thm:theorem-partition-generic-trace}
Let $A$ be a set which is partition generic in a non-trivial large $\Sigma^0_1$ class $\U \subseteq 2^\omega$. For every countable collection of non-empty $\Pi^0_1$ classes $\C_0, \C_1, \dots$ with no computable 1-trace. Then there is an infinite set $G \subseteq A$ such that none of the classes $\C_n$ admits a $G$-computable 1-trace.
\end{theorem}
\begin{proof}
Let $\F$ be a sufficiently generic filter containing $(\emptyset, A, \U)$.
By \Cref{lem:pg-forcing-infinite}, $G_\F$ is infinite. By construction, $G_\F \subseteq A$.
Last, by \Cref{lem:pg-forcing-non-trace}, none of the classes $\C_n$ admits a $G_\F$-computable 1-trace.
\end{proof}

\begin{corollary}[Liu~\cite{liu2015cone}]\label[corollary]{cor:liu-effective-closed-sets}
Let $\C_0, \C_1, \dots$ be a countable collection of non-empty $\Pi^0_1$ classes with no computable 1-trace.
For every set $A \subseteq \omega$, there is an infinite subset $H \subseteq A$ or $H \subseteq \overline{A}$ such that none of the classes $\C_n$ admits an $H$-computable 1-trace.
\end{corollary}
\begin{proof}
By \Cref{cor:partition-generic-regular}, either $A$ or $\overline{A}$ is partition generic in a non-trivial $\Sigma^0_1$ large class. Apply \Cref{thm:theorem-partition-generic-trace}.
\end{proof}

\begin{corollary}\label[corollary]{cor:kurtz-trace-avoidance}
Let $\C_0, \C_1, \dots$ be a countable collection of non-empty $\Pi^0_1$ classes with no computable 1-trace. For every Kurtz random $A \subseteq \omega$, there is an infinite subset $H \subseteq A$ such that none of the classes $\C_n$ admits an $H$-computable 1-trace.
\end{corollary}
\begin{proof}
By \Cref{prop:kurtz-bipg}, every Kurtz random is bi-partition generic.
Apply \Cref{thm:theorem-partition-generic-trace}.
\end{proof}

\begin{corollary}
For every Kurtz random~$A$, there is an infinite subset~$H \subseteq A$ which does not compute any set of positive Hausdorff dimension.
\end{corollary}
\begin{proof}
For every~$m,c$, let $\C_{m,c} = \{ X : \forall k K( X \uh_k ) \geq k/m-c\}$.
A set $X$ has positive Hausdorff dimension iff $X \in \C_{m,c}$ for some~$m,c \in \omega$.
For every~$m,c \in \omega$, $\C_{m,c}$ has no computable 1-trace (see Corollary~1.4 in \cite{kjoshanssen2020extracting}).
By \Cref{cor:kurtz-trace-avoidance}, there is an infinite subset~$H \subseteq A$
such that none of the classes~$\C_{m,c}$ admits an $H$-computable 1-trace.
In particular, $H$ does not compute any set of positive Hausdorff dimension.
\end{proof}

\subsection{Lowness}

It is clearly not the case that for every set~$A \subseteq \omega$, there is an infinite subset~$H \subseteq A$ or $H \subseteq \overline{A}$ of low degree. One can simply pick $A$ to be any bi-immune set relative to~$\emptyset'$. Then $A$ has no even $\Delta^0_2$ infinite subset in it or its complement. Therefore, there is no low or even $\Delta^0_2$ partition genericity subset basis theorem. More interestingly, Downey, Hirschfeldt, Lempp and Solomon~\cite{downey20010_2} constructed a $\Delta^0_2$ set with no low infinite subset in it or its complement, using a very involved infinite injury priority argument. Thus, there no hope of proving that every $\Delta^0_2$ partition generic sets has a low subset. On the other hand, if $A$ is a $\Delta^0_2$ set such that $\overline{A}$ is not partition generic, then $A$ is partition generic in a non-trivial $\Sigma^0_1$ partition large class in a strong sense, in which case $A$ admits an infinite low subset. 

\begin{theorem}\label[theorem]{thm:delta2-non-pg-low}
Let $A$ be a $\Delta^0_2$ set such that $\overline{A}$ is not partition generic.
Then there is an infinite subset $G \subseteq A$ of low degree.
\end{theorem}
\begin{proof}
Let $\bigcap_n \V_n$ be a $\Pi^0_2$ partition regular class such that $\overline{A} \not \in \bigcap_n \V_n$. In particular, there is some $m \in \omega$ such that $\overline{A} \not \in \V_m$.
Define a $\Delta^0_2$ decreasing sequence of Mathias conditions $(\sigma_0, X_0) \geq (\sigma_1, X_1) \geq \dots$ such that for every $s \in \omega$
\begin{itemize}
	\item[(1)] $\sigma_s \subseteq A$ ; $X_s$ is low ; $X_s \in \bigcap_n \V_n$ ;
	\item[(2)] $\sigma_{s+1} \in \V_e$ if $s = 2e$ ;
	\item[(3)] $(\sigma_{s+1}, X_{s+1}) \Vdash \Phi^G_e(e)\downarrow$ or $(\sigma_{s+1}, X_{s+1}) \Vdash \Phi^G_e(e)\uparrow$ if $s = 2e+1$
\end{itemize}
Start with the condition $(\emptyset, \omega)$.

\emph{Satisfying (2)}:
Given a condition $(\sigma_s, X_s)$ at a stage $s = 2e$, search $\emptyset'$-computably for some finite $\rho \subseteq X \cap A$ such that $\rho \in \V_s$. We claim that such a $\rho$ must exist. Indeed, since $X$ belongs to the partition regular class $\bigcap_n \V_n$, then either $X \cap A$ or $X \cap \overline{A}$ belongs to $\bigcap_n \V_n$. However, $\overline{A} \not \in \V_m$, so by upward-closure of partition regular classes, $X \cap \overline{A} \not \in \bigcap_n \V_n$, hence $X \cap A \in \bigcap_n \V_n$. The condition $(\sigma_s \cup \rho, X_s - \{0, \dots, \max \rho \})$ satisfies (2).

\emph{Satisfying (3)}:
Given an condition $(\sigma_s, X_s)$ at stage $s = 2e+1$, consider the $\Pi^{0,X_s}_1$ class $\C$ of all $B$ such that $\overline{B} \not \in \V_m$ and  $(\sigma_s, X_s \cap B) \Vdash \Phi^G_e(e)\uparrow$.
	Decide $\emptyset'$-computably if $\C$ is empty or not.
	If $\C = \emptyset$, then in particular $A \not \in \C$, so $(\sigma_s, X_s \cap A) \not\Vdash \Phi^G_e(e)\uparrow$. Search $\emptyset'$-computably for some $\rho \subseteq X_s \cap A$ such that $\Phi^{\sigma_s \cup \rho}_e(e)\downarrow$. The condition $(\sigma_s \cup \rho, X_s - \{0, \dots, \max \rho \})$ forces $\Phi^G_e(e)\downarrow$, hence satisfies (3).
	If $\C \neq \emptyset$, then by the low basis theorem, there is some $B \in \C$ of low degree. In particular, $\overline{B} \not \in \V_m$ so $\overline{B} \cap X_s \not \in \V_m$. However, $X_s \in \bigcap_n \V_n$, so by partition regularity of $\bigcap_n \V_n$, $B \cap X_s \in \bigcap_n \V_n$. The condition $(\sigma_s, X_s \cap B)$ forces $\Phi^G_e(e)\uparrow$, hence satisfies (3).
	
This completes the construction. Let $G = \bigcup_s \sigma_s$. In particular, by (1), $G \subseteq A$. By (2) $G \in \bigcap_n \V_n$, hence $G$ is infinite. By (3), $G' \leq_T \emptyset'$. This completes the proof.
\end{proof}

The previous theorem gives us a characterization of the $\Delta^0_2$ sets with no low infinite subset in them or their complements.

\begin{corollary}
A $\Delta^0_2$ set~$A$ has no low infinite subset in it or its complement iff it is bi-partition generic relative to every low set.
\end{corollary}
\begin{proof}
Suppose $A$ is bi-partition generic relative to every low set. By~\Cref{prop:bi-partition-generic-biimmune}, $A$ is bi-immune relative to every low set, hence has no low infinite subset in it or its complement.

Suppose $A$ is not bi-partition generic relative to a low set~$L$. By symmetry, say~$\overline{A}$ is not partition generic relative to~$L$. By~\Cref{thm:delta2-non-pg-low} relativized to~$L$, there is an infinite subset~$H \subseteq A$ such that~$H' \leq_T L' \leq_T \emptyset'$, hence $A$ has an infinite low subset.
\end{proof}

\begin{corollary}
There is a $\Delta^0_2$ set which is bi-partition generic relative to every low set.
\end{corollary}
\begin{proof}
Immediate by the previous corollary and the existence of a $\Delta^0_2$ set with no low infinite subset in either it or its complement~\cite{downey20010_2}.
\end{proof}

\section{Preservation of hyperimmunity}\label[section]{sect:pres-hyp}

The purpose of this section is to find the right genericity notion admitting a preservation of hyperimmunity genericity basis theorem, implying the preservation of hyperimmunity pigeonhole basis theorem. More precisely, the following theorem was proven by Patey~\cite{patey2015iterative}:

\begin{theorem}[Patey~\cite{patey2015iterative}]\label[theorem]{thm:preservation-nonce-pigeon-basis}
If $B$ is a hyperimmune set, then for every set $A \subseteq \omega$, there is an infinite set $H \subseteq A$ or $H \subseteq \overline{A}$ such that $B$ is $H$-hyperimmune.
\end{theorem}

Partition genericity is not the right notion to reprove this pigeonhole basis theorem. Indeed, the following proposition shows that there is no preservation of hyperimmunity partition genericity basis theorem.

\begin{proposition}\label[proposition]{prop:pg-non-basis-preservation-hyp}
There is a hyperimmune set~$B$ and a partition generic set~$A$ such that $B$ is hyperimmune relative to no infinite subset of~$A$.
\end{proposition}
\begin{proof}
Let~$A$ be a bi-hyperimmune set and let~$B = A$. Since $A$ is co-hyperimmune, by \Cref{prop:cohyp-pg}, $A$ is partition generic. Let~$H$ be an infinite subset of~$A$.
Then, $H$ is an infinite subset of~$B$, so $B$ is not $H$-hyperimmune.
\end{proof}

The remainder of this section is devoted to the design of a notion of genericity which admits a preservation of hyperimmunity basis theorem, implying \Cref{prop:pg-non-basis-preservation-hyp}.

\subsection{Hyperimmunity genericity}

From now on, fix a set~$B \subseteq \omega$. All the following definitions hold for any~$B$, but because of~\Cref{prop:omega-b-prtr-generic}, the only interesting case is when the set~$B$ is hyperimmune.

\begin{definition}
A class~$\L \subseteq 2^\omega \times 2^\omega$ is \emph{$\mathfrak{H}$-regular} if it is non-empty, upward-closed, and for every $(X,Y) \in \L$, every $k$, every $k$-cover $Z_0 \cup \dots \cup Z_{k-1} \supseteq X$ and every $\sigma \in 2^{<\omega}$, there is some $j < k$ such that $(Z_j, Y \setminus \sigma) \in \L$.
\end{definition}

\begin{proposition} \label{prop:prtr-unionlarge}
Suppose $\{\L_i\}_{i \in I}$ is an arbitrary non-empty collection of $\mathfrak{H}$-regular classes. Then $\bigcup_{i \in I} \L_i$ is $\mathfrak{H}$-regular.
\end{proposition}
\begin{proof}
It is clear that $\bigcup_{i \in I} \L_i$ is not empty. Let $(X_0,X_1) \in \bigcup_{i \in I} \L_i$. Let $Y_0 \supseteq X_0$ and $Y_1 \supseteq X_1$. There is some $i \in I$ such that $(X_0,X_1) \in \L_i$. As $\L_i$ is $\mathfrak{H}$-regular, $(Y_0,Y_1) \in \L_i \subseteq \bigcup_{i \in I} \L_i$.

Let $(X_0,X_1) \in \bigcup_{i \in I} \L_i$. Let $Y_0 \cup \dots \cup Y_k \supseteq X$ and $\sigma \in 2^{<\omega}$. There is some $i \in I$ such that $(X_0,X_1) \in \L_i$. As $\L_i$ is $\mathfrak{H}$-regular, $(Y_j,X_1 \setminus \sigma) \in \L_i \subseteq \bigcup_{i \in I} \L_i$ for some $j \leq k$.
\end{proof}

\begin{definition}
Given a class~$\A \subseteq 2^\omega \times 2^\omega$, let $\L_{\mathfrak{H}}(\A)$ denote the largest $\mathfrak{H}$-regular subclass of~$\A$. If $\A$ does not contain a $\mathfrak{H}$-regular class, let $\L_{\mathfrak{H}}(\A)$ be the empty set.
\end{definition}

\begin{proposition}\label{prop:prtr-complarge}
Let $\A \subseteq 2^\omega \times 2^\omega$ be an upward-closed class. Then 
$$\L_{\mathfrak{H}}(\A) = \left\{(X,Y) \in 2^\omega \times 2^\omega \ : \begin{array}{l}
 	\forall k \forall X_0 \cup \dots \cup X_k \supseteq X \\
 	\forall \sigma \in 2^{<\omega}\ \exists i \leq k\ (X_i, Y \setminus \sigma) \in \A
 \end{array}
 \right\}$$
\end{proposition}

\begin{definition}
A class~$\A \subseteq 2^\omega \times 2^\omega$ is \emph{$\mathfrak{H}$-large} if it is upward-closed and for every co-finite set~$X_1$, $\{ X_0 : (X_0, X_1) \in \A \}$ is partition large.
\end{definition}

\begin{proposition}\label[proposition]{prop:prtr-large-contains-prtr}
A class $\A \subseteq 2^\omega \times 2^\omega$ is $\mathfrak{H}$-large iff it is upward-closed and contains a $\mathfrak{H}$-regular subclass.
\end{proposition}
\begin{proof}
Suppose $\A$ is upward-closed and contains a $\mathfrak{H}$-regular subclass $\L \subseteq \A$.
Upward closure is trivially satisfied by hypothesis. By $\mathfrak{H}$-regularity of $\L$, $(\omega,\omega) \in \L$ and for every $k$-cover $Y_0, \dots, Y_{k-1}$ of $\omega$, and every~$\sigma \in 2^{<\omega}$, there is some $j < k$ such that $(Y_j, \omega \setminus \sigma) \in \L \subseteq \A$. So $\A$ is $\mathfrak{H}$-large.

Suppose now $\A$ is $\mathfrak{H}$-large. It is upward-closed by definition. We claim that $\L_{\mathfrak{H}}(\A)$ is $\mathfrak{H}$-regular. By definition of $\mathfrak{H}$-largeness and \Cref{prop:prtr-complarge}, $(\omega,\omega) \in \L_{\mathfrak{H}}(\A)$, so $\L_{\mathfrak{H}}(\A) \neq \emptyset$. By definition of $\L_{\mathfrak{H}}(\A)$, it is $\mathfrak{H}$-regular, hence $\A$ contains a $\mathfrak{H}$-regular subclass.
\end{proof}

\begin{definition}
Let $\A \subseteq 2^\omega \times 2^\omega$ be a class.
We say that $X$ is \emph{$\mathfrak{H}(B)$-generic} in $\A$ if $(X,\overline{B})$ belongs to every non-trivial $\Pi^0_2$ $\mathfrak{H}$-regular subclass of~$\A$. If $X$ is $\mathfrak{H}(B)$-generic in $2^\omega \times 2^\omega$, we simply say that $X$ is \emph{$\mathfrak{H}(B)$-generic}.
\end{definition}

\begin{proposition}\label[proposition]{prop:prtr-pg-invariant-superset}
If $X$ is $\mathfrak{H}(B)$-generic in $\A$	 and $Y \supseteq X$, then $Y$ is $\mathfrak{H}(B)$-generic in~$\A$.
\end{proposition}
\begin{proof}
Let $\L \subseteq \A$ be any non-trivial $\Pi^0_2$ $\mathfrak{H}$-regular subclass of $\A$.
Since $X$ is $\mathfrak{H}(B)$-generic in $\A$, then $(X,\overline{B}) \in \L$. Since $Y \supseteq X$, then by upward-closure of $\L$, $(Y, \overline{B}) \in \L$.
Therefore $Y$ is $\mathfrak{H}(B)$-generic in~$\A$.
\end{proof}

\begin{proposition}\label[proposition]{prop:prtr-pg-invariant-finite}
If $X$ is $\mathfrak{H}(B)$-generic in $\A$	 and $Y =^{*} X$, then $Y$ is $\mathfrak{H}(B)$-generic in~$\A$.
\end{proposition}
\begin{proof}
Let $\L \subseteq \A$ be any non-trivial $\Pi^0_2$ $\mathfrak{H}$-regular subclass of $\A$.
Since $X$ is $\mathfrak{H}(B)$-generic in $\A$, then $(X,\overline{B}) \in \L$. Since $Y =^{*} X$, then there is a finite set $F$ such that $Y \cup F \supseteq X$. By $\mathfrak{H}$-regularity of $\L$, either $(Y,\overline{B})$ or $(F,\overline{B})$ belongs to $\L$. Since $\L$ is non-trivial, $(F,\overline{B}) \notin \L$, so $(Y,\overline{B}) \in \L$.
Therefore $Y$ is $\mathfrak{H}(B)$-generic in~$\A$.
\end{proof}

\begin{lemma} \label{lemma:prtr-abslarge2}
Let $\A \subseteq 2^\omega \times 2^\omega$ be a class and $X$ be $\mathfrak{H}(B)$-generic in $\A$. For every $Y_0 \cup Y_1 \supseteq X$, if $(Y_0,\overline{B}) \notin \L_{\mathfrak{H}}(\A)$, then $Y_1$ is $\mathfrak{H}(B)$-generic in $\A$.
\end{lemma}
\begin{proof}
Suppose for contradiction that there is a non-trivial $\Pi^0_2$ $\mathfrak{H}$-regular class $\V \subseteq \A$ such that $(Y_1,\overline{B}) \notin \V$. In particular, $\V \subseteq \L_{\mathfrak{H}}(A)$, so $(Y_0,\overline{B}) \notin \V$. By partition $\mathfrak{H}$-regularity of $\V$, since $Y_0 \cup Y_1 \supseteq X$, then $(X,\overline{B}) \notin \V$, which contradicts $\mathfrak{H}(B)$-genericity of $X$ in $\A$.
\end{proof}

\begin{proposition} \label[proposition]{lemma:prtr-abslarge}
Let $\A \subseteq 2^\omega \times 2^\omega$ be a non-trivial $\mathfrak{H}$-large $\Sigma^0_1$ class. Suppose $X$ is $\mathfrak{H}(B)$-generic in $\A$. Let $Y_0 \cup \dots \cup Y_k \supseteq X$. Then there is a non-trivial $\mathfrak{H}$-large $\Sigma^0_1$ subclass $\U \subseteq \A$ together with some $i \leq k$ such that $Y_i$ is $\mathfrak{H}(B)$-generic in $\U$.
\end{proposition}
\begin{proof}
We proceed by induction on $k$. For $k = 0$, $Y_0 \supseteq X$, so by \Cref{prop:prtr-pg-invariant-superset}, $Y_0$ is $\mathfrak{H}(B)$-generic in $\A$. Take $\U = \A$ and we are done.

Suppose now that the property holds for $k-1$. Suppose $Y_k$ is not $\mathfrak{H}(B)$-generic in $\A$. 
Thus there is a non-trivial $\Pi^0_2$ $\mathfrak{H}$-regular class $\bigcap_{e \in C} \U_e \subseteq \A$ such that $(Y_k,\overline{B}) \not \in \bigcap_{e \in C} \U_e$. In particular, there is some $e \in \C$ such that $(Y_k,\overline{B}) \notin \U_e$. Note that $\bigcap_{e \in C} \U_e \subseteq \A \cap \U_e$, so $\A \cap \U_e$ is $\mathfrak{H}$-large. By \Cref{lemma:prtr-abslarge2}, $Y_0 \cup \dots \cup Y_{k-1}$ is $\mathfrak{H}(B)$-generic in $\U_e \cap \A$. By induction hypothesis on $\U_e \cap \A$ and $Y_0 \cup \dots \cup Y_{k-1}$, there is a non-trivial $\mathfrak{H}$-large $\Sigma^0_1$ class $\V \subseteq \U_e \cap \A$, together with some $i < k$ such that $Y_i$ is $\mathfrak{H}(B)$-generic in $\V$.
\end{proof}

\begin{proposition}\label[proposition]{prop:omega-b-prtr-generic}
Suppose~$B$ is hyperimmune. Then $\omega$ is $\mathfrak{H}(B)$-generic.
\end{proposition}
\begin{proof}
Let~$\L \subseteq 2^\omega \times 2^\omega$ be a non-trivial $\Pi^0_2$ $\mathfrak{H}$-regular class. We can assume $\L = \bigcap_n \U_n$ for some decreasing sequence of $\Sigma^0_1$ $\mathfrak{H}$-large classes. Fix~$n \in \omega$ and let $W \subseteq 2^{<\omega} \times 2^{<\omega}$ be a c.e. set such that $\U_n = \{ (X,Y) : \exists (\rho_0, \rho_1) \in W\ \rho_0 \subseteq X \wedge \rho_1 \subseteq Y \}$. Let~$P$ be a cofinite set.

Fix some~$s \in \omega$.
Since $\U_n$ is $\mathfrak{H}$-large, then $(\omega, \{s+1, s+2, \dots \}) \in \U_n$. Therefore, there is some $(\rho_0, \rho_1) \in W$ such that $\min \rho_1 > s$.
Let~$F_s = \rho_1$. Note that $F_s$ can be found computably uniformly in~$s$.
Since~$B$ is hyperimmune, there is some~$s$ such that $F_s \subseteq \overline{B}$. It follows that $(\omega, \overline{B}) \in \U_n$. Since this holds for every~$n$, then $(\omega, \overline{B}) \in \bigcap_n \U_n = \L$.
\end{proof}

\begin{proposition}\label[proposition]{prop:hyp-and-pg-rel-then-prtr-gen}
Suppose~$B$ is hyperimmune and $A$ is partition generic relative to~$B$.
Then $A$ is $\mathfrak{H}(B)$-generic.
\end{proposition}
\begin{proof}
Let~$\L \subseteq 2^\omega \times 2^\omega$ be a non-trivial $\Pi^0_2$ $\mathfrak{H}$-regular class. By \Cref{prop:omega-b-prtr-generic}, $(\omega, \overline{B}) \in \L$.
Since~$\L$ is $\mathfrak{H}$-regular, then for every~$k$, every $Z_0 \cup \dots \cup Z_{k-1} = \omega$, there is some~$j < k$ such that~$(Z_j, \overline{B}) \in L$. Thus the class $\L_0 = \{ X : (X, \overline{B}) \in \L \}$ is partition large. Moreover, $\L_0$ is $\Pi^0_2(B)$, so since~$A$ is partition generic relative to~$B$, $A \in \L_0$. It follows that $(A,\overline{B}) \in \L$.
\end{proof}

\subsection{Applications}

We now turn to the main application of $\mathfrak{H}(B)$-genericity, that is, a partition genericity subset basis theorem for preservation of hyperimmunity. As mentioned earlier, the basic  preservation of hyperimmunity statement fails for partition genericity. However, when one considers relativized partition genericity, then one can prove such as basis theorem.

\begin{theorem}\label[theorem]{thm:pg-tg-hyperimmune}
Suppose~$A$ is $\mathfrak{H}(B)$-generic in a non-trivial $\mathfrak{H}$-large $\Sigma^0_1$ class~$\U \subseteq 2^\omega \times 2^\omega$.
Then there is an infinite subset $H \subseteq A$ such that~$B$ is $H$-hyperimmune.
\end{theorem}
\begin{proof}
Consider the notion of forcing $(\sigma, X, \V)$ where $(\sigma, X)$ is as Mathias condition, 
and $\V \subseteq \U$ is a $\mathfrak{H}$-large $\Sigma^0_1$ class within which $X$ is $\mathfrak{H}(B)$-generic.
A condition $(\tau, Y, \W)$ \emph{extends} another condition $(\sigma, X, \V)$ if $(\tau, Y)$ Mathias extends $(\sigma, X)$ and $\W \subseteq \V$.

Any filter $\F$ induces a subset $G_\F = \bigcup_{(\sigma, X, \V) \in \F} \sigma$ of~$A$. 
The first property we prove is that $G_\F$ is an infinite set. 

\begin{lemma}\label[lemma]{lem:pg-tg-forcing-infinite}
Let $\F$ be a sufficiently generic filter. Then $G_\F$ is infinite.
\end{lemma}
\begin{proof}
Let us show that for any $n$, the set of conditions $(\tau, Y, \V)$ such that $\tau(x) = 1$ for some $x > n$ is dense.
Let $(\sigma, X, \V)$ be a condition. Since $X$ is $\mathfrak{H}(B)$-generic in $\V$ and $\V$ is non-trivial, then $X \in \L(\V)$, hence $X$ is infinite. Let $x \in X$ be greater than $n$, let $\tau$ be string $\sigma \cup \{x\}$ and let $Y = X \setminus \{0, \dots, |\rho|\}$. By \Cref{prop:prtr-pg-invariant-finite}, $Y$ is $\mathfrak{H}(B)$-generic in $\V$. Therefore, $(\tau, Y, \V)$ is a valid extension of $(\sigma, X, \V)$.
\end{proof}

In the following lemma, we interpret $\Phi_e^G$ as a partial $G$-c.e.\ array $\{ \Phi_e^G(n) : n \in \omega \}$. An array $\{F_n : n \in \omega \}$ \emph{intersects} a set~$C$ if $\forall n\  F_n \cap C \neq \emptyset$.

\begin{lemma}\label[lemma]{lem:pg-tg-forcing-requirement}
For every condition $(\sigma, X, \V)$, there is an extension $(\tau, Y, \W)$
forcing $\Phi_e^G$ not to be a $G$-c.e. array intersecting~$B$.	
\end{lemma}
\begin{proof}
We have two cases.

Case 1: the following class is $\mathfrak{H}$-large:
$$
\A = \{ (Y,C) \in \V : \exists \rho \subseteq Y \exists n\ \Phi_e^{\sigma \cup \rho}(n) \subseteq C \}
$$
Since $\A$ is a non-trivial $\Pi^0_2$ $\mathfrak{H}$-large subclass of~$\U$ and $X$ is $\mathfrak{H}(B)$-generic in~$\U$, then $(X,\overline{B}) \in \A$. Therefore there is some~$\rho \subseteq X$ and some~$n$ such that $\Phi_e^{\sigma \cup \rho}(n) \subseteq \overline{B}$. Let~$Y = X \setminus \{0, \dots, |\rho|\}$. By \Cref{prop:prtr-pg-invariant-finite}, $Y$ is $\mathfrak{H}(B)$-generic in~$\V$. Therefore $(\sigma \cup \rho, Y, \V)$ is an extension of~$(\sigma, X, \V)$ forcing $\Phi_e^G(n) \subseteq \overline{B}$.

Case 2: there is a co-finite set~$P$ and some~$k$ such that the following class is non-empty:
$$
\C = \left\{ Z_0 \oplus \dots \oplus Z_{k-1} : \begin{array}{l} 
 	Z_0 \cup \dots \cup Z_{k-1} = \omega \wedge \\
 	\forall i < k\ (Z_i,P) \not \in \V \vee \forall n \forall \rho \subseteq Z_i \Phi_e^{\sigma \cup \rho}(n)\uparrow \vee \Phi_e^{\sigma \cup \rho}(n) \not \subseteq P
 \end{array} \right\}
$$
Note that $\C$ is a non-empty $\Pi^0_1$ class. Pick any element $Z_0 \oplus \dots \oplus Z_{k-1} \in \C$. By \Cref{lemma:prtr-abslarge}, there is a $\mathfrak{H}$-large $\Sigma^0_1$ subclass~$\W \subseteq \V$ and some $i < k$ such that $Z_i \cap X$ is $\mathfrak{H}(B)$-generic in $\W$. In particular, $(Z_i \cap X, \overline{B}) \in \L_{\mathfrak{H}}(W) \subseteq \W$, so $(Z_i \cap X, \overline{B} \cap P) \in W \subseteq \V$ and by upward-closure of~$V$, $(Z_i, P) \in \V$, so for every~$n$ and every~$\rho \subseteq Z_i$, $\Phi_e^{\sigma \cup \rho}(n)\uparrow \vee \Phi_e^{\sigma \cup \rho}(n) \not \subseteq P$.
The condition $(\sigma, Z_i \cap X, \W)$ is an extension of~$(\sigma, X, \V)$ forcing~$\Phi_e^{G}$ not to be a c.e.\ array.
\end{proof}

We are now ready to prove \Cref{thm:pg-tg-hyperimmune}.
Let~$\F$ be a sufficiently generic filter containing $(\emptyset, A, \U)$. By \Cref{lem:pg-tg-forcing-infinite}, $G_\F$ is infinite. By construction, $G_\F \subseteq A$. Last, by \Cref{lem:pg-tg-forcing-requirement}, $B$ is $G_\F$-hyperimmune.
\end{proof}

\begin{corollary}\label[corollary]{hyp-and-pg-rel-then-preserve-hyp}
Suppose~$B$ is hyperimmune and $A$ is partition generic relative to~$B$.
Then there is an infinite subset $H \subseteq A$ such that~$B$ is $H$-hyperimmune.
\end{corollary}
\begin{proof}
By \Cref{prop:hyp-and-pg-rel-then-prtr-gen}, $A$ is $\mathfrak{H}(B)$-generic. By \Cref{thm:pg-tg-hyperimmune}, there is infinite subset $H \subseteq A$ such that~$B$ is $H$-hyperimmune.
\end{proof}

\begin{corollary}
Suppose~$B$ is hyperimmune and $A$ is Kurtz random relative to~$B$.
Then there is an infinite subset $H \subseteq A$ such that~$B$ is $H$-hyperimmune.
\end{corollary}
\begin{proof}
By~\Cref{prop:kurtz-bipg}, $A$ is partition generic relative to~$B$.
Apply \Cref{hyp-and-pg-rel-then-preserve-hyp}.
\end{proof}

\begin{corollary}
Let~$B$ be a hyperimmune set. Then for every set~$A \subseteq \omega$, there is an infinite set~$H \subseteq A$ or $H \subseteq \overline{A}$ such that $B$ is $H$-hyperimmune.
\end{corollary}
\begin{proof}
By \Cref{prop:omega-b-prtr-generic}, $\omega$ is $\mathfrak{H}(B)$-generic. By \Cref{lemma:prtr-abslarge}, there is a $\mathfrak{H}$-large $\Sigma^0_1$ class~$\U \subseteq 2^\omega \times 2^\omega$ within which either $A$ or $\overline{A}$ is $\mathfrak{H}$-generic. By \Cref{thm:pg-tg-hyperimmune}, there is an infinite set~$H \subseteq A$ or $H \subseteq \overline{A}$ such that $B$ is $H$-hyperimmune.
\end{proof}

\section{Preservation of non-$\Sigma^0_1$ definitions}\label[section]{sect:pres-sigma}

The purpose of this section is to find the right genericity notion admitting a preservation of non-$\Sigma^0_1$ definitions genericity basis theorem, implying the preservation of non-$\Sigma^0_1$ definitions pigeonhole basis theorem. More precisely, the following theorem was proven by Wang~\cite{Wang2016definability}:

\begin{theorem}[Wang~\cite{Wang2016definability}]\label[theorem]{thm:preservation-nonce-pigeon-basis}
If $B$ is a non-$\Sigma^0_1$ set, then for every set $A \subseteq \omega$, there is an infinite set $H \subseteq A$ or $H \subseteq \overline{A}$ such that $B$ is not $\Sigma^0_1(H)$.
\end{theorem}

Partition genericity is not the right notion to reprove this pigeonhole basis theorem. Indeed, the following proposition shows that there is no preservation of non-$\Sigma^0_1$ definitions partition genericity basis theorem.

\begin{proposition}\label[proposition]{prop:pg-non-basis-preservation-ce}
There is a non-$\Sigma^0_1$ set~$B$ and a partition generic set~$A$ such that $B$ is $\Sigma^0_1$ in every infinite subset of~$A$.	
\end{proposition}
\begin{proof}
Fix any computable linear order $L = (\omega, <_L)$ of order type~$\omega+\omega^{*}$ with no infinite computable ascending or descending sequence (see Rosenstein~\cite{rosenstein1982linear}). Let~$A$ be the $\omega$ part of this order. Note that $\overline{A}$ is the $\omega^{*}$ part of this order. 
First, note that $A$ is bi-hyperimmune. Indeed, suppose $F_0, F_1, \dots$ is a c.e. array such that for every~$n$, $F_n \cap A \neq \emptyset$. Then the set~$X = \{ \min_{<_L} F_n : n \in \omega \}$ is an infinite c.e. subset of~$A$. By thinning out the set~$X$, one can compute an infinite increasing sequence, contradicting our assumption. Thus $A$ is hyperimmune.
By a symmetric argument, $A$ is co-hyperimmune.

By \Cref{prop:cohyp-pg}, since~$A$ is co-hyperimmune, it is partition generic. Let~$B = A$. Since~$A$ is hyperimmune, then it is not c.e. Let~$H$ be any infinite subset of~$A$. 
Then $B = \{ x \in \omega : \exists n \in H,\ x <_L n \}$. Thus $B$ is c.e. in every infinite subset of~$A$.
\end{proof}

The remainder of this section is devoted to the design of a notion of genericity which admits a preservation of non-c.e. definitions basis theorem, implying \Cref{thm:preservation-nonce-pigeon-basis}.

\subsection{Enumeration genericity}

From now on, fix a set~$B \subseteq \omega$. All the following definitions hold for any~$B$, but because of~\Cref{prop:pren-pg-large-ce}, the only interesting case is when the set~$B$ is non-$\Sigma^0_1$. Contrary to hyperimmunity genericity, we need to work with largeness rather than regularity, for some reasons which will become clear when proving \Cref{prop:pren-pg-large-ce}.

\begin{definition}
A class $\A \subseteq 2^\omega \times \omega$ is \emph{$\mathfrak{E}(B)$-large} if
\begin{enumerate}
\item If $(X,n) \in \A$ and $X \subseteq Y$, then $(Y,n) \in \A$
\item For every finite set~$F \subseteq B$ and every $k$-cover $Y_0, \dots, Y_{k-1}$ of $\omega$, there is some~$j < k$ such that $\forall n \in F\ (Y_j, n) \in \A$. 
\end{enumerate}
\end{definition}

The definition of $\mathfrak{E}(B)$-largeness ensures that the class $\{ X : \forall n \in B\ (X,n) \in \A \}$ is partition large. If~$B$ is not c.e. but $\A$ is $\Sigma^0_1$, then there must be some \qt{overflow}, in the sense that for every~$k$-cover $Y_0, \dots, Y_{k-1}$ of $\omega$, there must be some~$n \in \overline{B}$ and some~$j < k$ such that $(Y_j,n) \in \A$. Such an~$n$ however depends on~$k$. Therefore, there does not exist in general an~$n \in \overline{B}$ such that the class $\{ X : (X,n) \in \A \}$ is partition large, but one still have the following proposition:

\begin{proposition}\label[proposition]{prop:pren-pg-large-ce}
Suppose $B$ is not c.e. 
For every $\mathfrak{E}(B)$-large $\Sigma^0_1$ class~$\U \subseteq 2^\omega \times \omega$, the class $\{ X : \forall n \in B\ (X, n) \in \U \wedge 
\exists n \in \overline{B}\ (X, n) \in \U \}$ is partition large.
\end{proposition}
\begin{proof}
Suppose $B$ is non-c.e. Let $\U \subseteq 2^\omega \times \omega$ be any non-trivial $\mathfrak{E}(B)$-large $\Sigma^0_1$ class. For every $k, s$, let $A_{k,s} = \{ m : \forall Z_0 \cup \dots \cup Z_{k-1} = \omega\ \exists i < k\ (Z_i, m) \in \U \wedge \forall n \in B \uh_s\ (Z_i, n) \in \U \}$. Since~$\U$ is $\mathfrak{E}(B)$-large, for every $k, s$, the set $A_{k,s}$ is a c.e. superset of~$B$, so since~$B$ is not c.e., there is some~$n_s \in A_{k,s} \setminus B$. By the infinite pigeonhole principle, for every $Z_0 \cup \dots \cup Z_{k-1} = \omega$, there is some~$i < k$ such that for infinitely many~$s$, $n_s \in Z_i$ and $\forall n \in B \uh_s\ (Z_i, n) \in \U$.
It follows that for every $Z_0 \cup \dots \cup Z_{k-1} = \omega$, there is some~$i < k$
such that $\exists n \in \overline{B}\ (Z_i,n) \in \U$ and $\forall n \in B\ (Z_i,n) \in \U$.
Thus the class~$\{ X : \forall n \in B\ (X, n) \in \U  \wedge \exists n \in \overline{B}\ (X, n) \in \U \}$ is partition large.
\end{proof}


A class~$\A \subseteq 2^\omega \times \omega$ is \emph{non-trivial} if for every~$(X,n) \in \A$, $|X| \geq 2$. 

\begin{definition}
Let $\A \subseteq 2^\omega \times \omega$ be a class.
We say that $X$ is \emph{$\mathfrak{E}(B)$-generic} in $\A$ if for every non-trivial $\mathfrak{E}(B)$-large $\Sigma^0_1$ class~$\U \subseteq \A$, $X \in \L(\{ Z : \forall n \in B\ (Z, n) \in \U \wedge \exists n \in \overline{B}\ (Z, n) \in \U \})$. If $X$ is $\mathfrak{E}(B)$-generic in $2^\omega \times \omega$, we simply say that $X$ is \emph{$\mathfrak{E}(B)$-generic}.
\end{definition}

\begin{proposition}\label[proposition]{prop:ce-and-pg-rel-then-pren-gen}
Suppose~$B$ is not c.e. and $A$ is partition generic relative to~$B$.
Then $A$ is $\mathfrak{E}(B)$-generic.
\end{proposition}
\begin{proof}
Fix a non-trivial $\mathfrak{E}(B)$-large $\Sigma^0_1$ class~$\U \subseteq 2^\omega \times \omega$.
Since~$B$ is not c.e., by \Cref{prop:pren-pg-large-ce}, the class $\V = \{ X : \forall n \in B\ (X, n) \in \U \wedge \exists n \in \overline{B}\ (X, n) \in \U \}$ is partition large. Note that $\V$ is $\Pi^0_2(B)$.
By \Cref{prop:large-contains-pr}, $\V$ contains a partition regular subclass, so $\L(\V)$ is partition regular. By \Cref{prop:complarge}, $\L(\V)$ is $\Pi^0_2(B)$, so since $A$ is partition generic relative to~$B$, $A \in \L(\V)$.
\end{proof}

\begin{proposition}\label[proposition]{prop:pren-pg-omega-ce}
Suppose $B$ is not c.e. Then $\omega$ is $\mathfrak{E}(B)$-generic.	
\end{proposition}
\begin{proof}
Immediate from \Cref{prop:ce-and-pg-rel-then-pren-gen}, since $\omega$ is partition generic relative to~$B$.
\end{proof}

\begin{proposition}\label[proposition]{prop:pren-pg-infinite}
If $X$ is $\mathfrak{E}(B)$-generic in a non-trivial $\mathfrak{E}(B)$-large $\Sigma^0_1$ class $\A \subseteq 2^\omega \times \omega$, then $X$ is infinite.
\end{proposition}
\begin{proof}
Let~$\V = \{ Z : \forall n \in B\ (Z, n) \in \U \wedge \exists n \in \overline{B}\ (Z, n) \in \A \}$. Since $X$ is $\mathfrak{E}(B)$-generic in~$\A$, $X \in \L(\V)$. Since $\A$ is non-trivial, for every~$X \in \V$, $|X| \geq 2$. Therefore, by \Cref{prop:non-trivial-u2}, $\L(\V)$ is non-trivial, so since~$X \in \L(\V)$, $X$ is infinite.
\end{proof}

\begin{proposition}\label[proposition]{prop:pren-pg-invariant-superset}
If $X$ is $\mathfrak{E}(B)$-generic in $\A$	 and $Y \supseteq X$, then $Y$ is $\mathfrak{E}(B)$-generic in~$\A$.
\end{proposition}
\begin{proof}
Let $\U$ be a non-trivial $\mathfrak{E}(B)$-large $\Sigma^0_1$ subclass of~$\A$. Since $X$ is $\mathfrak{E}(B)$-generic in $\A$, then, letting~$\V = \{ Z : \forall n \in B\ (Z, n) \in \U \wedge \exists n \in \overline{B}\ (Z, n) \in \U \}$, $X \in \L(\V)$. Since $\L(\V)$ is closed upward, $Y \in \L(\V)$.
Therefore $Y$ is $\mathfrak{E}(B)$-generic in~$\A$.
\end{proof}

\begin{proposition}\label[proposition]{prop:pren-pg-invariant-finite}
If $X$ is $\mathfrak{E}(B)$-generic in $\A$	and $Y =^{*} X$, then $Y$ is $\mathfrak{E}(B)$-generic in~$\A$.
\end{proposition}
\begin{proof}
Let $\U \subseteq \A$ be any non-trivial $\mathfrak{E}(B)$-large  $\Sigma^0_1$ subclass of $\A$.
Let~$\V = \{ Z : Z : \forall n \in B\ (Z, n) \in \U \wedge \exists n \in \overline{B}\ (Z, n) \in \U \}$. Since $X$ is $\mathfrak{E}(B)$-generic in $\A$, then $X \in \L(\V)$. 
Since $Y =^{*} X$, then there is a finite set $F$ such that $Y \cup F \supseteq X$.
By partition regularity of~$\L(\V)$, either $F \in \L(\V)$, or $Y \in \L(\V)$. 
Since~$\U$ is non-trivial, then $\L(\V)$ is non-trivial, so $F \not \in \L(\V)$. It follows that $Y \in \L(\V)$.
Therefore $Y$ is $\mathfrak{E}(B)$-generic in~$\A$.
\end{proof}

\begin{lemma} \label{lemma:pren-abslarge2}
Suppose $X$ is $\mathfrak{E}(B)$-generic in a class~$\A \subseteq 2^\omega \times \omega$. For every $Y_0 \cup Y_1 \supseteq X$, if $Y_0 \not \in \L(\{ Z : \forall n \in B\ (Z, n) \in \U \wedge \exists n \in \overline{B}\ (Z, n) \in \A \})$, then $Y_1$ is $\mathfrak{E}(B)$-generic in~$\A$.
\end{lemma}
\begin{proof}
Fix any non-trivial $\mathfrak{E}(B)$-large $\Sigma^0_1$ class $\U \subseteq \A$ and let~$\V = \{ Z : \forall n \in B\ (Z, n) \in \U \wedge \exists n \in \overline{B}\ (Z, n) \in \U \}$. 
Since $X$ is $\mathfrak{E}(B)$-generic in~$\A$, then $X \in \L(\V)$. Since $Y_0 \cup Y_1 \supseteq X$, then by partition regularity of $\L(\V)$, either~$Y_0 \in \L(\V)$, or $Y_1 \in \L(\V)$. Since~$\U \subseteq \A$, then $\V \subseteq \{ Z : \forall n \in B\ (Z, n) \in \U \wedge \exists n \in \overline{B}\ (Z, n) \in \A \}$, so $\L(\V) \subseteq \L(\{ Z : \forall n \in B\ (Z, n) \in \U \wedge \exists n \in \overline{B}\ (Z, n) \in \A \})$. If follows from hypothesis that~$Y_0 \not \in \L(\V)$, so $Y_1 \in \L(\V)$.
Therefore, $Y_1$ is $\mathfrak{E}(B)$-generic in~$\A$.
\end{proof}

\begin{proposition} \label[proposition]{lemma:pren-abslarge}
Let $\A \subseteq 2^\omega \times \omega$ be a non-trivial $\mathfrak{E}(B)$-large $\Sigma^0_1$ class. Suppose $X$ is $\mathfrak{E}(B)$-generic in $\A$. Let $Y_0 \cup \dots \cup Y_k \supseteq X$. Then there is a non-trivial $\mathfrak{E}(B)$-large $\Sigma^0_1$ subclass $\U \subseteq \A$ together with some $i \leq k$ such that $Y_i$ is $\mathfrak{E}(B)$-generic in $\U$.
\end{proposition}
\begin{proof}
We proceed by induction on $k$. For $k = 0$, $Y_0 \supseteq X$, so by \Cref{prop:pren-pg-invariant-superset}, $Y$ is $\mathfrak{E}(B)$-generic in $\A$. Take $\U = \A$ and we are done.
Suppose now that the property holds for $k-1$. Suppose $Y_k$ is not $\mathfrak{E}(B)$-generic in $\A$. 
Thus there is a non-trivial $\mathfrak{E}(B)$-large $\Sigma^0_1$ class $\U \subseteq \A$ such that $Y_k \not\in \L(\{ Z : \forall n \in B\ (Z, n) \in \U \wedge \exists n \in \overline{B}\ (Z, n) \in \U \})$.
By \Cref{lemma:pren-abslarge2}, $Y_0 \cup \dots \cup Y_{k-1}$ is $\mathfrak{E}(B)$-generic in $\U$. By induction hypothesis on $\U$ and $Y_0 \cup \dots \cup Y_{k-1}$, there is a non-trivial $\mathfrak{E}(B)$-large $\Sigma^0_1$ subclass $\V \subseteq \U$ together with some $i < k$ such that $Y_i$ is $\mathfrak{E}(B)$-generic in $\V$.
\end{proof}

\subsection{Applications}

We now turn to the main application of enumeration genericity, which is a partition genericity subset basis theorem for preservation of non-$\Sigma^0_1$ definitions.

\begin{theorem}\label[theorem]{thm:pg-eg-ce}
Suppose~$A$ is $\mathfrak{E}(B)$-generic in a non-trivial $\mathfrak{E}(B)$-large $\Sigma^0_1$ class~$\U \subseteq 2^\omega \times \omega$.
Then there is an infinite subset $H \subseteq A$ such that~$B$ is not $\Sigma^0_1(H)$.
\end{theorem}
\begin{proof}
Consider the notion of forcing $(\sigma, X, \V)$ where $(\sigma, X)$ is as Mathias condition, 
and $\V \subseteq \U$ is a $\mathfrak{E}(B)$-large $\Sigma^0_1$ class within which $X$ is $\mathfrak{E}(B)$-generic.
A condition $(\tau, Y, \W)$ \emph{extends} another condition $(\sigma, X, \V)$ if $(\tau, Y)$ Mathias extends $(\sigma, X)$ and $\W \subseteq \V$.

Any filter $\F$ induces a subset $G_\F = \bigcup_{(\sigma, X, \V) \in \F} \sigma$ of~$A$. 
The first property we prove is that $G_\F$ is an infinite set. 

\begin{lemma}\label[lemma]{lem:pg-eg-forcing-infinite}
Let $\F$ be a sufficiently generic filter. Then $G_\F$ is infinite.
\end{lemma}
\begin{proof}
Let us show that for any $n$, the set of conditions $(\tau, Y, \V)$ such that $\tau(x) = 1$ for some $x > n$ is dense.
Let $(\sigma, X, \V)$ be a condition. Since $X$ is $\mathfrak{E}(B)$-generic in $\V$, then by \Cref{prop:pren-pg-infinite}, $X$ is infinite. Let $x \in X$ be greater than $n$, let $\tau$ be string $\sigma \cup \{x\}$ and let $Y = X \setminus \{0, \dots, |\rho|\}$. By \Cref{prop:pren-pg-invariant-finite}, $Y$ is $\mathfrak{E}(B)$-generic in $\V$. Therefore, $(\tau, Y, \V)$ is a valid extension of $(\sigma, X, \V)$.
\end{proof}

\begin{lemma}\label[lemma]{lem:pg-eg-forcing-requirement}
For every condition $(\sigma, X, \V)$, there is an extension $(\tau, Y, \W)$ forcing $W_e^G \neq B$.	
\end{lemma}
\begin{proof}
We have two cases.

Case 1: the following class is $\mathfrak{E}(B)$-large:
$$
\A = \{ (Y,n) \in \V : \exists \rho \subseteq Y\ n \in W_e^{\sigma \cup \rho} \}
$$
Since $\A$ is a non-trivial $\mathfrak{E}(B)$-large $\Sigma^0_1$ subclass of~$\U$ and $X$ is $\mathfrak{E}(B)$-generic in~$\U$, then $X \in \L(\{ Z : \forall n \in B\ (Z, n) \in \U \wedge \exists n \in \overline{B}\ (Z,n) \in \A \})$. Therefore there is some~$\rho \subseteq X$ and some~$n \in \overline{B}$ such that $n \in W_e^{\sigma \cup \rho}$. Let~$Y = X \setminus \{0, \dots, |\rho|\}$. By \Cref{prop:pren-pg-invariant-finite}, $Y$ is $\mathfrak{E}(B)$-generic in~$\V$. Therefore $(\sigma \cup \rho, Y, \V)$ is an extension of~$(\sigma, X, \V)$ forcing $W_e^G \neq B$.

Case 2: there is some finite set~$F \subseteq B$ and some~$k \in \omega$ such that the following class is non-empty:
$$
\C = \left\{ Z_0 \oplus \dots \oplus Z_{k-1} : \begin{array}{l}
 		Z_0 \cup \dots \cup Z_{k-1} = \omega \wedge\\
 		\forall i < k\ \exists n \in F\ (Z_i,n) \not \in \V \vee \forall \rho \subseteq Z_i\ n \not \in W_e^{\sigma \cup \rho}
 \end{array}\right\}
$$
Note that $\C$ is a non-empty $\Pi^0_1$ class. Pick any element $Z_0 \oplus \dots \oplus Z_{k-1} \in \C$. By \Cref{lemma:pren-abslarge}, there is a non-trivial $\mathfrak{E}(B)$-large $\Sigma^0_1$ subclass $\W$ of~$\U$ and some $i < k$ such that $Z_i \cap X$ is $\mathfrak{E}(B)$-generic in $\W$. In particular, $\forall n \in B\ (Z_i \cap X, n) \in \W \subseteq \V$, and by upward-closure of $\V$, $\forall n \in B\ (Z_i, n) \in \V$. By choice of~$Z_i$, there is some~$n \in F$ such that $(Z_i,n) \not \in \V \vee \forall \rho \subseteq Z_i\ n \not \in W_e^{\sigma \cup \rho}$. Since $F \subseteq B$, $(Z_i, n) \in \V$, so $\forall \rho \subseteq Z_i\ n \not \in W_e^{\sigma \cup \rho}$. The condition $(\sigma, Z_i \cap X, \W)$ is an extension of~$(\sigma, X, \V)$ forcing~$n \not \in W_e^G$, with~$n \in B$.
\end{proof}

We are now ready to prove \Cref{thm:pg-eg-ce}.
Let~$\F$ be a sufficiently generic filter containing $(\emptyset, A, \U)$. By \Cref{lem:pg-eg-forcing-infinite}, $G_\F$ is infinite. By construction, $G_\F \subseteq A$. Last, by \Cref{lem:pg-eg-forcing-requirement}, $B$ is not $\Sigma^0_1(G_\F)$.
\end{proof}

\begin{corollary}\label[corollary]{cor:non-ce-pg-preservation-nonce}
Let~$B$ be a non-$\Sigma^0_1$ set. Let~$A$ be partition generic relative to~$B$. 
Then there is an infinite subset~$H \subseteq A$ such that~$B$ is not $\Sigma^0_1(H)$.
\end{corollary}
\begin{proof}
By~\Cref{prop:ce-and-pg-rel-then-pren-gen}, $A$ is $\mathfrak{E}(B)$-generic.
By \Cref{thm:pg-eg-ce}, there is an infinite subset~$H \subseteq A$ such that~$B$ is not $\Sigma^0_1(H)$.
\end{proof}

\begin{corollary}
Let~$B$ be a non-$\Sigma^0_1$ set. Let~$A$ be Kurtz random relative to~$B$.
Then there is an infinite subset~$H \subseteq A$ such that~$B$ is not $\Sigma^0_1(H)$.
\end{corollary}
\begin{proof}
By \Cref{prop:kurtz-bipg}, $A$ is partition generic relative to~$B$, so by \Cref{cor:non-ce-pg-preservation-nonce}, there is an infinite subset~$H \subseteq A$ such that~$B$ is not $\Sigma^0_1(H)$.
\end{proof}

\begin{proof}[Proof of~\Cref{thm:preservation-nonce-pigeon-basis}]
Let~$B$ be a non-$\Sigma^0_1$ set and $A$ be a set.
By \Cref{prop:pren-pg-omega-ce}, $\omega$ is $\mathfrak{E}(B)$-generic.
By \Cref{lemma:pren-abslarge}, there is a non-trivial $\mathfrak{E}(B)$-large $\Sigma^0_1$ class~$\A \subseteq 2^\omega \times \omega$ within which either~$A$ or $\overline{A}$ is $\mathfrak{E}(B)$-generic. By \Cref{thm:pg-eg-ce}, there is an infinite subset~$H \subseteq A$ such that~$B$ is not $\Sigma^0_1(H)$.
\end{proof}


\section{Partition genericity and computability}\label[section]{sect:questions}

We conclude this study of partition regularity and partition genericity by considering the corresponding notions of lowness, and constructing a partition generic set which is both computably dominated and of non-DNC degree.

\begin{definition}\ 
\begin{enumerate}
	\item A set~$X$ is \emph{low for partition genericity} if every partition generic set is partition $X$-generic.
	\item A set $X$ is \emph{low for partition regularity} if for every $\Pi^0_2(X)$ partition regular class $\L$, there is a $\Pi^0_2$ partition regular subclass $\M \subseteq \L$
	\item A set~$X$ is \emph{low for partition largeness} if for every $\Sigma^0_1(X)$ partition large class~$\U$, there is a $\Sigma^0_1$ large subclass~$\V \subseteq \U$.
\end{enumerate}
\end{definition}

It is clear that if~$X$ is low for partition regularity or low for partition largeness, then it is low for partition genericity. Actually, all three notions are trivial, in the sense that the only degree which is low for partition genericity is the computable one.

\begin{proposition}\label[proposition]{prop:low-pg-computable}
If $X$ is low for partition genericity, then $X$ is computable.
\end{proposition}
\begin{proof}
Let~$X$ be a non-computable set. Let~$A = \{ \sigma \in 2^{<\omega} : \sigma \prec X \}$. Then $A \equiv_T X$ and every infinite subset of~$A$ computes~$A$.
Either $A$ is partition generic in a non-trivial partition large $\Sigma^0_1$ class, or ~$\overline{A}$ is partition generic in $2^\omega$. In the first case, by \Cref{thm:pg-basis-cone-avoidance}, there is an infinite subset of~$A$ which does not compute~$A$, contradiction. Therefore~$\overline{A}$ is partition generic.

Let us show that~$\overline{A}$ is not partition generic relative to~$X$.
Since~$A$ is infinite, $\L_A$ is a non-trivial $\Pi^0_2(X)$ partition regular class.
However, $\overline{A} \not \in \L_A$. Therefore, $X$ is not low for partition genericity.
\end{proof}

The degrees of partition generic sets are not fully understood. In \Cref{subsect:pg}, we proved that every co-hyperimmune set and every Kurtz random is partition generic. By the computably dominated basis theorem, there are Kurtz randoms of computably dominated degree. Since every weakly 1-generic is Kurtz random, there are Kurtz randoms of non-DNC degree. On the other hand, Stephan and Yu~\cite{stephan2006lowness} proved that the degrees which are low for Kurtz-randomness are precisely the computably dominated non-DNC degrees, which implies that no Kurtz random is of both degrees simultaneously. We now prove the existence of a partition generic set which is low for Kurtz-randomness. This is done using a perfect tree forcing starting from a suitable tree.

\begin{definition}[Terwijn and Zambella~\cite{terwijn1997algorithmic}]
Fix a canonical coding of all finite sets $D_0, D_1, \dots$.
A set~$X$ is \emph{computably traceable} if there is a computable function $p$ such that, for each function $f \leq_T A$, there is a computable function $h$ satisfying, for all $n$,
$|D_{h(n)}| \leq p(n)$ and $f(n) \in D_{h(n)}$.
\end{definition}

Terwijn and Zambella~\cite{terwijn1997algorithmic} proved that the computably traceable degrees are precisely those which are low for Schnorr randomness. It is clear that every computably traceable set is computably dominated. Moreover, by Kjos-Hanssen, Merkle, and Stephan~\cite{kjoshanssen2011kolmogorov}, every computably traceable set is of non-DNC degree.

\begin{proposition}\label[proposition]{prop:pg-ct}
There is a partition generic set which is computably traceable and of minimal degree.
\end{proposition}
\begin{proof}
A \emph{function tree} is a function $T : 2^{<\omega} \to 2^{<\omega}$ such that for every $\sigma \in 2^{<\omega}$, $T(\sigma 0)$ and $T(\sigma 1)$ are incompatible extensions of~$T(\sigma)$. 
Any such function $T : 2^{<\omega} \to 2^{<\omega}$ induces a function $T : 2^\omega \to 2^\omega$ by defining $T(X) = \bigcup_{\sigma \prec X} T(\sigma)$. We then write $[T] = \{ f(X) : X \in 2^\omega \}$. A function tree~$S$ \emph{extends} a function tree~$T$ (written $S \leq T$) if $[S] \subseteq [T]$.

\begin{lemma}
There is a computable function tree~$T_0$ such that for every $X,Y \in 2^\omega$ with $X \neq Y$,
then $T_0(X) \cup T_0(Y) =^{*} \omega$.
\end{lemma}
\begin{proof}
Let~$T_0(\epsilon) = \epsilon$. Suppose $T_0$ is defined on $2^{\leq n}$ for some~$n$.
Let~$\sigma_0, \sigma_1, \dots, \sigma_{2^{n+1}-1}$ be the list of all strings of length~$n+1$.
For every~$i < 2^{n+1}$, let $\tau_i$ be the string of length~$2^{n+1}$ which has a 0 at position $i$, and 1 everywhere else. Let~$T_0(\sigma_i) = T_0(\sigma_i \uh_n)^\frown \tau_i$.
For instance, $T_0(0) = 01$, $T_0(1) = 10$, $T(00) = 010111$, $T(01) = 011011$, $T(10) = 101101$, $T(11) = 101110$. Note that every two strings of same length is sent to strings of same length.

Let~$X, Y \in \omega$ be such that $X \neq Y$, and let $\sigma$ be the longest common substring.
We claim that for every~$n > |T_0(\sigma)$, then either~$n \in T_0(X)$, or $n \in T_1(Y)$.
Indeed, let $t$ be the smallest length such that $n < |T_0(X \uh_t)|$, or equivalently such that $n < |T_0(Y \uh_t)|$. Let~$i$ and $j < 2^t$ be such that~$X\uh_t$ and $Y\uh_t$ are respectively the $i$th and the $j$th string of length~$t$. Note that since $n > |T_0(\sigma)|$, then $X\uh_t$ and $Y\uh_t$  are incomparable, hence~$i \neq j$.
By definition, $T_0(X \uh_t) = T_0(X \uh_{t-1})^\frown \tau_i$ and $T_0(Y \uh_t) = T_0(Y \uh_{t-1})^\frown \tau_j$, where $\tau_i$ and $\tau_j$ do not have a 0 at the same position. 
By choice of~$t$, $n \geq |T_0(X \uh_{t-1})| = |T_0(Y \uh_{t-1})|$, so either $n \in T_0(X \uh_t)$, or $n \in T_0(Y \uh_t)$.
\end{proof}

Consider the notion of forcing $\mathbb{P}$ whose conditions are computable function trees extending $T_0$. Any sufficiently generic filter~$\F$ induces a set~$G_\F$ which is the unique member of $\bigcap_{T \in \F} [T]$.
A condition $T$ \emph{forces} a formula $\varphi(G)$ if the formula hods for every~$G \in [T]$.

\begin{lemma}\label[lemma]{lem:pg-ct-progress}
For every non-trivial partition large $\Sigma^0_1$ class~$\U \subseteq 2^\omega$ and every condition~$T$ there is an extension~$S \leq T$ forcing~$G \in \U$.
\end{lemma}
\begin{proof}
Pick two $X,Y \in 2^\omega$ with $X \neq Y$.
Since~$T \leq T_0$, then $T(X) \cup T(Y) =^{*} \omega$. Since~$\U$ is a non-trivial partition large class, either $T(X) \in \U$, or $T(Y) \in \U$. Assume the first case holds, by symmetry. Since $\U$ is $\Sigma^0_1$, there there is a finite string $\sigma \prec X$ such that $[T(\sigma)] \subseteq \U$. Let~$S$ be the extension of~$T$ defined by $S(\tau) = T(\sigma\tau)$. Then for every~$G \in [S]$, $G \in \U$.
\end{proof}

Let~$\F$ be a sufficiently generic filter. By \Cref{lem:pg-ct-progress}, $G_\F$ is partition generic. It is well known that every sufficiently generic filter for computable Sacks forcing produces sets of minimal degree. Terwijn and Zambella~\cite{terwijn1997algorithmic} proved that these sets are also computably traceable.
\end{proof}

\vspace{0.5cm}

\bibliographystyle{plain}
\bibliography{bibliography}

\begin{thebibliography}{10}

\bibitem{dorais2012variant}
Fran\c{c}ois~G. Dorais.
\newblock A variant of {M}athias forcing that preserves {ACA${}_0$}.
\newblock {\em Arch. Math. Logic}, 51(7-8):751--780, 2012.

\bibitem{downey20010_2}
Rod Downey, Denis~R. Hirschfeldt, Steffen Lempp, and Reed Solomon.
\newblock A {$\Delta^0_2$} set with no infinite low subset in either it or its
  complement.
\newblock {\em Journal of Symbolic Logic}, 66(3):1371--1381, 2001.

\bibitem{Dzhafarov2009Ramseys}
Damir~D. Dzhafarov and Carl~G. Jockusch.
\newblock Ramsey's theorem and cone avoidance.
\newblock {\em Journal of Symbolic Logic}, 74(2):557--578, 2009.

\bibitem{greenberg2009lowness}
Noam Greenberg and Joseph~S. Miller.
\newblock {Lowness for Kurtz randomness}.
\newblock {\em Journal of Symbolic Logic}, 74(2):665--678, 2009.

\bibitem{Hirschfeldt2008strength}
Denis~R. Hirschfeldt, Carl~G. Jockusch, Bj{\o}rn Kjos-Hanssen, Steffen Lempp,
  and Theodore~A. Slaman.
\newblock {The strength of some combinatorial principles related to {R}amsey's
  theorem for pairs}.
\newblock {\em Computational Prospects of Infinity, Part II: Presented Talks,
  World Scientific Press, Singapore}, pages 143--161, 2008.

\bibitem{Jockusch197201}
Carl~G. Jockusch and Robert~I. Soare.
\newblock {$\Pi^0_1$} classes and degrees of theories.
\newblock {\em Transactions of the American Mathematical Society}, 173:33--56,
  1972.

\bibitem{kjoshanssen2011strong}
Bj\o~rn Kjos-Hanssen.
\newblock A strong law of computationally weak subsets.
\newblock {\em J. Math. Log.}, 11(1):1--10, 2011.

\bibitem{kjoshanssen2020extracting}
Bj\o~rn Kjos-Hanssen and Lu~Liu.
\newblock Extracting randomness within a subset is hard.
\newblock {\em Eur. J. Math.}, 6(4):1438--1451, 2020.

\bibitem{kjoshanssen2009infinite}
Bj{\o}rn Kjos-Hanssen.
\newblock Infinite subsets of random sets of integers.
\newblock {\em Mathematics Research Letters}, 16:103--110, 2009.

\bibitem{kjoshanssen2011kolmogorov}
Bj{\o}rn Kjos-Hanssen, Wolfgang Merkle, and Frank Stephan.
\newblock Kolmogorov complexity and the recursion theorem.
\newblock {\em {Transactions of the American Mathematical Society}},
  363(10):5465--5480, 2011.

\bibitem{Liu2012RT22}
Lu~Liu.
\newblock {RT$^2_2$ does not imply WKL$_0$}.
\newblock {\em Journal of Symbolic Logic}, 77(2):609--620, 2012.

\bibitem{liu2015cone}
Lu~Liu.
\newblock Cone avoiding closed sets.
\newblock {\em Transactions of the American Mathematical Society},
  367(3):1609--1630, 2015.

\bibitem{mileti2004partition}
Joseph~Roy Mileti.
\newblock {\em Partition theorems and computability theory}.
\newblock ProQuest LLC, Ann Arbor, MI, 2004.
\newblock Thesis (Ph.D.)--University of Illinois at Urbana-Champaign.

\bibitem{monin2021weakness}
Benoit Monin and Ludovic Patey.
\newblock The weakness of the pigeonhole principle under hyperarithmetical
  reductions.
\newblock {\em J. Math. Log.}, 21(3):Paper No. 2150013, 41, 2021.

\bibitem{patey2015iterative}
Ludovic Patey.
\newblock Iterative forcing and hyperimmunity in reverse mathematics.
\newblock In Arnold Beckmann, Victor Mitrana, and Mariya Soskova, editors, {\em
  CiE. Evolving Computability}, volume 9136 of {\em Lecture Notes in Computer
  Science}, pages 291--301. Springer International Publishing, 2015.

\bibitem{patey2017iterative}
Ludovic Patey.
\newblock Iterative forcing and hyperimmunity in reverse mathematics.
\newblock {\em Computability}, 6(3):209--221, 2017.

\bibitem{rosenstein1982linear}
Joseph~G. Rosenstein.
\newblock {\em Linear orderings}, volume~98 of {\em Pure and Applied
  Mathematics}.
\newblock Academic Press, Inc. [Harcourt Brace Jovanovich, Publishers], New
  York-London, 1982.

\bibitem{stephan2006lowness}
Frank Stephan and Liang Yu.
\newblock Lowness for weakly 1-generic and {K}urtz-random.
\newblock In {\em Theory and applications of models of computation}, volume
  3959 of {\em Lecture Notes in Comput. Sci.}, pages 756--764. Springer,
  Berlin, 2006.

\bibitem{terwijn1997algorithmic}
Sebastiaan~A Terwijn and Domenico Zambella.
\newblock Algorithmic randomness and lowness.
\newblock 1997.

\bibitem{Wang2016definability}
Wei Wang.
\newblock The definability strength of combinatorial principles.
\newblock {\em J. Symb. Log.}, 81(4):1531--1554, 2016.

\end{thebibliography}

\end{document}